\documentclass[11pt, a4paper]{article}
\usepackage{amsmath,bm,amssymb,latexsym,color,enumerate,extarrows}
\usepackage{authblk}
\usepackage[mathscr]{eucal}
\usepackage[colorlinks,
            linkcolor=blue,
            anchorcolor=green,
            citecolor=magenta
           ]{hyperref}
\newtheorem{thm}{Theorem}[section]
\newtheorem{pro}[thm]{Proposition}
\newtheorem{lem}[thm]{Lemma}

\newtheorem{con}{Construction}

\newtheorem{fact}{Fact}
\newtheorem{rema}[thm]{Remark}

\newtheorem{example}{Example}[section]
\newtheorem{defin}{Definition}[section]

\newcommand{\proof}{{\it Proof.\quad}}
\newcommand{\qed}{\hfill\Box\medskip}

\allowdisplaybreaks
\usepackage{CJK}

\begin{document}

\renewcommand{\baselinestretch}{1.3}

\title{\bf On extremal cross $t$-intersecting families with $t$-covering number conditions}

\author[1]{Yu Zhu\thanks{E-mail: \texttt{zhu\_y@mail.bnu.edu.cn}}}
\author[1]{Benjian Lv\thanks{Corresponding author. E-mail: \texttt{bjlv@bnu.edu.cn}}}
\author[1]{Kaishun Wang\thanks{E-mail: \texttt{wangks@bnu.edu.cn}}}

\affil[1]{\small Laboratory of Mathematics and Complex Systems (Ministry of Education), School of Mathematical Sciences, Beijing Normal University, Beijing 100875, China}

\date{}
\maketitle
\begin{abstract}
Let $n$, $k$ and $t$ be positive integers, and let $\mathcal{F}$ be a collection of $k$-subsets of $[n]=\{1,2,\dots,n\}$.
The $t$-covering number $\tau_t(\mathcal{F})$ of $\mathcal{F}$ is defined as the minimum size of a set $T$ such that $|F\cap T|\geq t$ for all $F\in \mathcal{F}$.
For positive integers $k_1$ and $k_2$, let $\mathcal{F}_i$ be a collection of $k_i$-subsets of $[n]$ for $i\in \{1,2\}$.
The families $\mathcal{F}_1$ and $\mathcal{F}_2$ are said to be cross $t$-intersecting if $|F_1\cap F_2|\geq t$ for all $F_1\in\mathcal{F}_1$ and $F_2\in \mathcal{F}_2$. When $\mathcal{F}_1=\mathcal{F}_2$, $\mathcal{F}_1$ is called a $t$-intersecting family.
In this paper, we first characterize the extremal structures of cross $t$-intersecting families $\mathcal{F}_1$ and $\mathcal{F}_2$ that maximize $|\mathcal{F}_1||\mathcal{F}_2|$ under the condition that $\tau_t(\mathcal{F}_1)\geq t+1$ and $\tau_t(\mathcal{F}_2)\geq t+1$. We then describe the maximal $t$-intersecting families with $t$-covering number $t+1$.

\medskip



\end{abstract}

\section{Introduction}
Let $n$, $k$, $\ell$ and $t$ be positive integers. Write $[\ell,k]=\{\ell,\ell+1,\ldots,k\}$ with $\ell\leq k$, and abbreviate $[1,k]$ as $[k]$.
Given $S \subseteq M \subseteq [n]$, denote by $[S, M]_k$ the family of all $k$-subsets of $M$ that contain $S$, in particular, write $[S\rangle_k=\big[S, [n]\big]_k$ and ${M\choose k}=[\emptyset, M]_k$.
For a family $\mathcal{F}\subseteq{[n]\choose k}$, a subset $T\subseteq[n]$ is called a $t$-\emph{cover} of $\mathcal{F}$ if $|T\cap F|\geq t$ for all $F\in\mathcal{F}$, and the $t$-\emph{covering number} $\tau_t(\mathcal{F})$ of $\mathcal{F}$ is the minimum size of a $t$-cover of $\mathcal{F}$.
Obviously, $\tau_t(\mathcal{F})=t$ if and only if all members of $\mathcal{F}$ share a common $t$-subset.

We call $\mathcal{F}\subseteq {[n]\choose k}$ and $\mathcal{G}\subseteq {[n]\choose \ell}$ \emph{cross $t$-intersecting} if $|F\cap G|\geq t$ for all $F\in \mathcal{F}$ and $G\in \mathcal{G}$, and simply say cross intersecting when $t=1$.
They are said to be \emph{maximal} if $\mathcal{F}=\mathcal{F}'$ and $\mathcal{G}=\mathcal{G}'$ for every pair of cross $t$-intersecting families $\mathcal{F}'\subseteq {[n]\choose k}$ and $\mathcal{G}'\subseteq {[n]\choose \ell}$ such that $\mathcal{F}\subseteq\mathcal{F}'$ and $\mathcal{G}\subseteq\mathcal{G}'$.
Clearly, $t\leq \tau_t(\mathcal{F})\leq \ell$ and $t\leq \tau_t(\mathcal{G})\leq k$ when both $\mathcal{F}$ and $\mathcal{G}$ are non-empty.
For two pairs of cross $t$-intersecting families $(\mathcal{F},~\mathcal{G})$ and $(\mathcal{F}',~ \mathcal{G}')$,  we say that they are \emph{isomorphic} if there exists a permutation $\sigma\in S_n$ such that $(\sigma(\mathcal{F}),~ \sigma(\mathcal{G}))=(\mathcal{F}',~\mathcal{G}')$, denoted by $(\mathcal{F},~\mathcal{G}) \simeq (\mathcal{F}',~ \mathcal{G}')$.

If the cross $t$-intersecting families $\mathcal{F}$ and $\mathcal{G}$ satisfy $\mathcal{F}=\mathcal{G}$, then $\mathcal{F}$ is well known as a \emph{$t$-intersecting family}.
The study of $t$-intersecting families began with the Erd\H{o}s-Ko-Rado theorem \cite{Erdos-Ko-Rado-1961-313},
which shows that for large $n$, the size of a $t$-intersecting family $\mathcal{F}\subseteq {[n]\choose k}$ is at most ${n-t\choose k-t}$, and equality holds if and only if $\mathcal{F}=[T\rangle_k$ for some $t$-subset $T$. Observe that the $t$-covering number of this extremal family  is $t$.
Another foundational problem in this area is to determine the extremal structures of $t$-intersecting families $\mathcal{F}$ with $\tau_t(\mathcal{F})\geq t+1$. This problem was solved by Hilton and Milner \cite{Hilton-Milner-1967} for $t=1$, and  by Frankl \cite{Frankl-1978-1} and Ahlswede and Khachatrian \cite{Ahlswede-Khachatrian-1996} for $t\geq 2$.
For further studies on $t$-intersecting families, we refer readers to \cite{Ahlswede-Khachatrian-1997,Cao-set,Frankl-1978, Frankl-Furedi-1986, Frankl--Furedi-1991, Han-Kohayakawa, Kostochka-Mubayi, Wilson-1984, Yao-Liu-Wang-2026}.

The problem of determining the maximum product of sizes of cross $t$-intersecting families $\mathcal{F}$ and $\mathcal{G}$ has received some attention.
It was first investigated by
Pyber \cite{Pyber-1986}, who showed that $|\mathcal{F}||\mathcal{G}|\leq {n-1\choose k-1}{n-1\choose \ell-1}$ when $t=1$ and $n$ is large. Subsequently,  Matsumoto and Tokushige \cite{Matsumoto-Tokushige-1989} found the least possible $n$ for which the upper bound holds.
For $t\geq 2$ and $k=\ell$, Tokushige \cite{Tokushige-2013}  conjectured that $|\mathcal{F}||\mathcal{G}|\leq {n-t\choose k-t}^2$ for $n\geq (t+1)(k-t+1)$, and if $n>   (t+1)(k-t+1)$, then equality holds if and only if
$\mathcal{F}=\mathcal{G}=[T\rangle_k$ for some $t$-subset $T$.
After a series of contributions \cite{Tokushige-2010, Frankl--Lee--Siggers--Tokushige-2014} toward the conjecture, it was completely settled very recently by Zhang and Wu \cite{Wang-Wu-2025} for $t\geq 3$, and by Tanaka and Tokushige \cite{Tanaka-Tokushige-2025} for $t=2$. For general $k$ and $\ell$, Tokushige \cite{Tokushige-2013} also conjectured that $|\mathcal{F}||\mathcal{G}|\leq {n-t\choose k-t}{n-t\choose \ell-t}$ for $n\geq (t+1)\big(\max\{k,\ell\}-t+1\big)$. Borg \cite{Borg-2014-uniform, Borg-2016} proved the conjecture for large $n$ and showed that equality holds if and only if $\mathcal{F}=[T\rangle_k$ and $\mathcal{G}=[T\rangle_\ell$ for some $t$-subset $T$.

Over the past decade, the maximum product of sizes of cross $t$-intersecting families under restricted conditions and their extremal structures have attracted considerable attention.
Cao et al.~\cite{Cao-Lu-Lv-Wang-2024} and He et al.~\cite{He-Li-Wu-Zhang-2026} determined the maximum value of $|\mathcal{F}||\mathcal{G}|$ with $\tau_t(\mathcal{F})\geq t$ and $\tau_t(\mathcal{G})\geq t+1$ for large $n$, and described their extremal structures.
Frankl and Wang \cite{Frankl-Wang-2022, Frankl-Wang-2023} characterized the maximum value of $|\mathcal{F}||\mathcal{G}|$ under $k=\ell$,  $\tau_t(\mathcal{F})\geq t+1$ and $\tau_t(\mathcal{G})\geq t+1$ for large $n$.
More related papers can be found in \cite{Borg-2014, Frankl-- Kupavskii-2017}.

In this paper, we focus on the extremal structures of cross $t$-intersecting families $\mathcal{F}\subseteq {[n]\choose k}$ and $\mathcal{G}\subseteq {[n]\choose \ell}$ under the conditions  $\tau_t(\mathcal{F})\geq t+1$ and $\tau_t(\mathcal{G})\geq t+1$, where $k$ and $\ell$ are not necessarily equal, with $\min\{k,\ell\}\geq t+1$.
First, we define the following constructions.
\begin{con}\label{Con1}
$\mathcal{A}(k,t)=\left\{F\in {[n]\choose k} : |F\cap [t+2]|\geq t+1\right\}$.
\end{con}
\begin{con}\label{Con2}
$\mathcal{B}(k;(a_1,a_2,a_3,a_4))=\big[\{a_1,a_2\}\big\rangle_k \bigcup \big[\{a_2,a_3\}\big\rangle_k\bigcup \big[\{a_3,a_4\}\big\rangle_k,$ where $a_1$, $a_2$, $a_3$ and $a_4$ are distinct elements in $[n]$.
\end{con}
\begin{con}\label{Con3}
\begin{align*}
\mathcal{C}_1(\ell,t)&=\big\{[\ell+1]\setminus \{i\} : i\in  [t+1]\big\}\bigcup \big[[t+1]\big\rangle_\ell,\\
\mathcal{C}_2(k,t;\ell)&=\left\{F\in {[n]\choose k} : |F\cap [t+1]|= t,~|F\cap [t+2,\ell+1]|\geq 1\right\}\bigcup \big[[t+1]\big\rangle_k.
\end{align*}
\end{con}

\begin{con}\label{Con4}
Let $X\in {[t+1,n]\choose k-t+1}$ and $Y\in {[t+1,n]\choose \ell-t+1}$ such that $|X\cap Y|\geq 2-\delta_{1,t}$, where $\delta_{1,t}=1$ if $t=1$, and $0$ otherwise. Define
$$\mathcal{H}(k,t;X,Y)=\left\{F\in {[n]\choose k} : [t]\subseteq F,~|F\cap Y|\geq 1\right\}\bigcup \big\{X\cup[t]\setminus \{i\} : i\in [t]\big\}.$$
\end{con}

Note that $\mathcal{H}(t+1,t;X,X)\simeq \mathcal{A}(t+1,t)$, $\mathcal{H}(2,1;X,Y)\simeq \mathcal{B}(2;(1,2,3,4))$ if $|X\cap Y|=1$, and $(\mathcal{C}_1(t+1,t),~\mathcal{C}_2(k,t;t+1))\simeq (\mathcal{A}(t+1,t),~\mathcal{A}(k,t))$.
Clearly, the pairs $(\mathcal{A}(k,t),~\mathcal{A}(\ell,t))$, $(\mathcal{H}(k,t;X,Y),~\mathcal{H}(\ell,t;Y,X))$, and $(\mathcal{C}_1(\ell,t),~\mathcal{C}_2(k,t;\ell))$  are cross $t$-intersecting families. In addition, $\mathcal{B}(k;(a_1,a_2,a_3,a_4))$ and $\mathcal{B}(\ell;(a_1,a_3,a_2,a_4))$ form cross intersecting families.

If we further view $|\mathcal{A}(k,t)||\mathcal{A}(\ell,t)|$, $|\mathcal{H}(k,t;X,Y)|| \mathcal{H}(\ell,t;Y,X)|$, $|\mathcal{C}_1(\ell,t)| |\mathcal{C}_2(k,t;\ell)|$ and $|\mathcal{B}(k;(1,3,2,4))| |\mathcal{B}(\ell;(1,2,3,4))|$ (with $t=1$ in the last case) as polynomials in $n$, by Remark \ref{remark-1}, then it is straightforward to calculate that their leading terms are
$c\cdot ((k-t-1)!(\ell-t-1)!)^{-1} n^{k+\ell-2t-2},$ where the constants $c$ are $(t+2)^2$, $(k-t+1)(\ell-t+1)$,  $(t+1)(\ell-t)+1$ and $9$, respectively. Notably, for almost all $k$, $\ell$ and $t$, one can determine which pair of these cross $t$-intersecting families attains the largest product of their sizes when $n$ is sufficiently large. However, this becomes extremely difficult for small $n$.

Our first result characterizes the extremal structures of cross $t$-intersecting families, whose $t$-covering numbers are at least $t+1$.

\begin{thm}\label{main-1}
Let $n$, $k_1$, $k_2$ and $t$ be positive integers with $k_1\geq k_2\geq t+1$ and $n\geq (t+1)^2(k_1+k_2)^2(k_1-t+1)(k_2-t+1)+k_1+k_2-t$. Suppose that $\mathcal{F}_1\subseteq {[n]\choose k_1}$ and $\mathcal{F}_2\subseteq {[n]\choose k_2}$ are cross $t$-intersecting families with $\tau_t(\mathcal{F}_1) \geq t+1$ and $\tau_t(\mathcal{F}_2) \geq t+1$, such that $|\mathcal{F}_1||\mathcal{F}_2|$ is maximized.
\begin{enumerate}[{\rm(i)}]
\item \sloppy{If $t\geq 2$, then $(\mathcal{F}_1,~\mathcal{F}_2)$ $\simeq$ $(\mathcal{A}(k_1,t),~\mathcal{A}(k_2,t))$, $(\mathcal{C}_1(k_1,t),~\mathcal{C}_2(k_2,t;k_1))$, or $(\mathcal{H}(k_1,t;X,Y),~\mathcal{H}(k_2,t;Y,X))$ with $X\in {[t+1,n]\choose k_1-t+1}$, $Y\in {[t+1,n]\choose k_2-t+1}$ and $|X\cap Y|\geq 2$.}
  \item If $t=1$, then $(\mathcal{F}_1,~\mathcal{F}_2)$ $\simeq$ $(\mathcal{A}(k_1,1),~\mathcal{A}(k_2,1))$, $(\mathcal{B}(k_1;(1,3,2,4)), ~\mathcal{B}(k_2;(1,2,3,4)))$, or $(\mathcal{H}(k_1,1;X,Y),~\mathcal{H}(k_2,1;Y,X))$ with $X\in {[2,n]\choose k_1}$, $Y\in {[2,n]\choose k_2}$ and $|X\cap Y|\geq 1$.
\end{enumerate}
\end{thm}

To prove Theorem \ref{main-1}, in Section \ref{2}, we focus on studying the structures of cross $t$-intersecting families $\mathcal{F}_1$ and $\mathcal{F}_2$ with $\tau_t(\mathcal{F}_1)=\tau_t(\mathcal{F}_2)=t+1$. Note that if $\mathcal{F}_1=\mathcal{F}_2=\mathcal{F}$, then $\mathcal{F}$ is a $t$-intersecting family. By applying Lemmas \ref{m} and \ref{equivalent}, we characterize all the maximal $t$-intersecting families $\mathcal{F}$ with $\tau_t(\mathcal{F})=t+1$. To this end, we introduce the following notation.
Let $r$, $k_1$, $k_2$, \dots, $k_r$ be positive integers. We say that $\mathcal{F}_1\subseteq {[n]\choose k_1}$,~$\mathcal{F}_2\subseteq {[n]\choose k_2}$,~$\dots$,~ $\mathcal{F}_r\subseteq {[n]\choose k_r}$ are \emph{pairwise cross $t$-intersecting} if $|F_i\cap F_j|\geq t$ for all $F_i\in \mathcal{F}_i$ and $F_j\in \mathcal{F}_j$ with distinct $i,j\in [r]$.
They are said to be \emph{maximal} if $\mathcal{F}_1=\mathcal{F}'_1$, $\mathcal{F}_2=\mathcal{F}'_2$, \dots, $\mathcal{F}_r=\mathcal{F}'_r$ for every collection of pairwise cross $t$-intersecting families $\mathcal{F}'_1$, $\mathcal{F}'_2$, $\dots$, $\mathcal{F}'_r$ such that $\mathcal{F}_1\subseteq\mathcal{F}'_1$, $\mathcal{F}_2\subseteq\mathcal{F}'_2$, $\dots$, $\mathcal{F}_r\subseteq\mathcal{F}'_r$. When $t=1$, we simply say that they are pairwise cross intersecting.
Related studies on pairwise cross $t$-intersecting families can be found in \cite{Borg-2009, Hilton-1977, Huang-Peng-2025, Li-Zhang-2025, Shi-Frankl-Qian-2022, Wang-Zhang-2011, Xi-Kong-Ge-2024}. We now present our second main result.

\begin{thm}\label{main-2}
Let $n$, $k$ and $t$ be positive integers with $n\geq k\geq t+2$. Let $\mathcal{F}\subseteq{[n]\choose k}$ be a maximal $t$-intersecting family such that $\tau_t(\mathcal{F})=t+1$. Then $\mathcal{F}$ is isomorphic to one of
the following families:
\begin{enumerate}[{\rm(i)}]
  \item $\mathcal{F}^{\prime}=\mathcal{A}(k,t)$;
  \item $\mathcal{F}'=\mathcal{H}(k,t; [t+1,k+1],[t+1,k+1])$;
  \item $\mathcal{F}^{\prime}= \bigcup\limits_{i\in[t+1]} \left\{A\cup ([t+1]\backslash \{i\}) : A\in \mathcal{A}_i\right\}\bigcup \big[[t+1]\big\rangle_k$, where $\mathcal{A}_i\subseteq {[t+2,n]\choose k-t}$, $i=1,\dots, t+1$ are maximal pairwise cross intersecting with at least two of them being non-empty;
  \item $\mathcal{F}^{\prime}=\left\{F\in {[n]\choose k} : [t]\subseteq F,~|F\cap [m]|\geq t+1\right\}
      \bigcup\left\{A\cup[t] : A\in\mathcal{A}\right\}\bigcup\big\{B\cup\big([m]\setminus \{i\}\big) : i\in [t],~B\in\mathcal{B}\big\},$ where $m\in [t+2,k]$, and $\mathcal{A} \subseteq {[m+1,n]\choose k-t}$ and $\mathcal{B}\subseteq {[m+1,n]\choose k-m+1}$ are a pair of maximal cross intersecting families with $\mathcal{B}\neq \emptyset$.
\end{enumerate}
\end{thm}

The rest of this paper is organized as follows.
In Section \ref{2}, we study cross
$t$-intersecting families $\mathcal{F}$ and $\mathcal{G}$ with $\tau_t(\mathcal{F})=\tau_t(\mathcal{G})=t+1$, characterizing their structures or estimating the product of their sizes. In Section \ref{2.2}, we prove Theorems \ref{main-1} and \ref{main-2}. In Section \ref{3}, we establish several inequalities used throughout the paper.

\section{$\tau_t(\mathcal{F})=\tau_t(\mathcal{G})=t+1$} \label{2}

In this section, we always assume that $n$, $k$, $\ell$ and $t$ are positive integers satisfying $\min\{k, \ell\}\geq t+1$ and $n\geq (t+1)^2(k+\ell)^2(k-t+1)(\ell-t+1)+k+\ell-t$.
Let $\mathcal{F} \subseteq \binom{[n]}{k}$ and $\mathcal{G} \subseteq \binom{[n]}{\ell}$ be maximal cross $t$-intersecting families with $\tau_t(\mathcal{F})=\tau_t(\mathcal{G})=t+1$. Let $\mathcal{T}_f$ and $\mathcal{T}_g$ denote the collections of all $t$-covers of $\mathcal{F}$ and $\mathcal{G}$ with size $t+1$, respectively.
Denote $M_f = \cup_{A \in \mathcal{T}_f} A$ and $M_g = \cup_{B \in \mathcal{T}_g} B$.
Based on the maximality of $\mathcal{F}$ and $\mathcal{G}$, it is clear that
\begin{align}\label{abs}
\bigcup_{B\in \mathcal{T}_g}[B\rangle_k\subseteq \mathcal{F}~\text{and}~\bigcup_{A\in \mathcal{T}_f}[A\rangle_\ell\subseteq \mathcal{G}.
\end{align}
For $S \subseteq [n]$, define $\mathcal{F}_S=\big\{F\in \mathcal{F} : S\subseteq F\big\}$.
The following facts will be used frequently.
\begin{fact}{\rm (\cite[Corollary~9.1.4]{BCN})}\label{DRG}
If $\mathcal{T}\subseteq {[n]\choose t+1}$ is a maximal t-intersecting family, then $\mathcal{T}\simeq {[t+2]\choose t+1}$ or $\mathcal{T}\simeq\big[[t]\big\rangle_{t+1}$.
\end{fact}

\begin{fact}{\rm(\cite[Lemma~2.1]{Cao-Lu-Lv-Wang-2024})} \label{s-tt}
  $\mathcal{T}_f$ and $\mathcal{T}_g$ are cross $t$-intersecting families.
\end{fact}

Let $n$, $x$, $y$, $m$ and $t$ be positive integers. Write
\begin{align}
g(m,x,y,t)&=m{n-t-1\choose x-t-1}+y(t+1)(y-t+1){n-t-2\choose x-t-2},\label{equ-3}\\
a(x,t)&=(t+2){n-t-2\choose x-t-1}+{n-t-2\choose x-t-2}, \label{equ-1}\\
c_2(x,y,t)&=(t+1)\left({n-t-1\choose x-t}-{n-y-1\choose x-t}\right)+{n-t-1\choose x-t-1},\label{equ-11}\\
c_1(y,t)&={n-t-1\choose y-t-1}+t+1,\label{equ-12}\\
h(x,y,t)&={n-t\choose x-t}-{n-y-1\choose x-t}+t. \label{equ-2}
\end{align}

\begin{rema}\label{remark-1}
Direct computation gives that $|\mathcal{A}(k,t)|=a(k,t)$, $|\mathcal{B}(k;(a_1,a_2,a_3,a_4))|=a(k,1)$, $|\mathcal{C}_2(k,t;\ell)|=c_2(k,\ell,t)$, $|\mathcal{C}_1(\ell,t)|=c_1(\ell,t)$,
and $|\mathcal{H}(k,t;X,Y)| = h(k,\ell,t)$ in Constructions \ref{Con1}--\ref{Con4}.
\end{rema}

First, we establish upper bounds for $|\mathcal{F}|$ and $|\mathcal{G}|$, whose leading coefficients depend on $|\mathcal{T}_g|$ and $|\mathcal{T}_f|$, respectively.

\begin{lem}\label{not-contain}
We have $|\mathcal{F}|\leq g(|\mathcal{T}_g|,k,\ell,t)$ and $
|\mathcal{G}|\leq g(|\mathcal{T}_f|,\ell,k,t).$
\end{lem}
\proof By the symmetry of $\mathcal{F}$ and $\mathcal{G}$, we only prove $|\mathcal{F}|\leq g(|\mathcal{T}_g|,k,\ell,t)$.
Let $\mathcal{F}^{\prime}$ be the subfamily of $\mathcal{F}$ consisting of those members that contain no element of $\mathcal{T}_g$. From \eqref{abs}, it is obvious that $\mathcal{F}=\big(\cup_{B\in \mathcal{T}_g}\mathcal{F}_B\big)\cup~ \mathcal{F}'$, and so
\begin{align}\label{F'}
|\mathcal{F}|\leq |\mathcal{T}_g|{n-t-1\choose k-t-1}+ |\mathcal{F}^{\prime}|.
\end{align}
If $\mathcal{F}'=\emptyset$, the inequality holds trivially. In the following, assume that $\mathcal{F}'\neq\emptyset$. If $k=t+1$, then $\mathcal{F}=\mathcal{T}_g$, and so
$\mathcal{F}^{\prime}=\emptyset$, a contradiction. Hence $k\geq t+2$.

Let $A_0\in \mathcal{T}_f$.
Then
$\mathcal{F}^{\prime}= \cup_{U\in {A_0\choose t}}\mathcal{F}^{\prime}_U$.
Let $U_1\in{A_0\choose t}$ with $|\mathcal{F}^{\prime}_{U_1}|= \max\big\{|\mathcal{F}^{\prime}_{U}| : U\in {A_0\choose t}\big\}$. Then
$$|\mathcal{F}^{\prime}|\leq (t+1)|\mathcal{F}^{\prime}_{U_1}|.$$

Since $\tau_t(\mathcal{G})= t+1$, there exists $G_1\in \mathcal{G}$ with $|G_1\cap U_1|\leq t-1$. For each $F\in \mathcal{F}^{\prime}_{U_1}$, it follows from $|F\cap G_1|\geq t$ that $|F\cap (U_1\cup G_1)|\geq t+1$. Then
$$\mathcal{F}^{\prime}_{U_1}\subseteq \bigcup_{R\in {U_1\cup G_1\choose t+1},~U_1\subseteq R}\mathcal{F}^{\prime}_R.$$
Let $U_2\in{U_1\cup G_1\choose t+1}$ such that $U_1\subseteq U_2$ and $|\mathcal{F}^{\prime}_{U_2}|=\max\big\{|\mathcal{F}^{\prime}_R| : R\in {U_1\cup G_1\choose t+1},~U_1\subseteq R\big\}$. Then
$$|\mathcal{F}^{\prime}_{U_1}|\leq \ell|\mathcal{F}^{\prime}_{U_2}|.$$

Since every member of $\mathcal{F}'$ avoids all sets in $\mathcal{T}_g$, we have $U_2\notin \mathcal{T}_g$, and then there exists $G_2\in \mathcal{G}$ with $|G_2\cap U_2|=w\leq t-1$. It follows from \cite[Lemma~2.7]{Cao-Lu-Lv-Wang-2024} and Lemma \ref{notknow}~(i) that
$$|\mathcal{F}^{\prime}_{U_2}|\leq {\ell-w\choose t-w}{n-(t+1)-t+w\choose k-(t+1)-t+w}\leq (\ell-t+1){n-t-2\choose k-t-2}.$$
Therefore
$$|\mathcal{F}^{\prime}|\leq \ell(t+1)(\ell-t+1){n-t-2\choose k-t-2},$$
and so
$|\mathcal{F}|\leq
g(|\mathcal{T}_g|,k,\ell,t)$
from \eqref{F'}, which completes the proof.$\qed$

As shown in Lemma \ref{not-contain}, $|\mathcal{T}_g|$ directly determines the leading term of the inequality. In particular, when $n$ is sufficiently large, the leading term dominates the bound in magnitude, and the second term becomes negligible by comparison. Thus, $|\mathcal{T}_g|$ plays a central role in determining the upper bound of  $|\mathcal{F}|$.

Next, we divide the discussion into the following cases: both $\mathcal{T}_f$ and $\mathcal{T}_g$ are $t$-intersecting families; exactly one of them is $t$-intersecting; and neither of them is $t$-intersecting.
By Fact \ref{s-tt}, we have $\tau_t(\mathcal{T}_f),\tau_t(\mathcal{T}_g)\in \{t,t+1\}$, and both $\mathcal{T}_f$ and $\mathcal{T}_g$ are $t$-intersecting if and only if $\mathcal{T}_f\cup \mathcal{T}_g$ is $t$-intersecting.

\begin{pro}\label{l+1_t+1_t}
If $\mathcal{T}_f$ is $t$-intersecting and $\ell=t+1$, then $\tau_t(\mathcal{T}_f)=t+1$, and $\mathcal{T}_g$ is $t$-intersecting.
\end{pro}
\proof Since $\ell=\tau_t(\mathcal{G})$, the maximality of $\mathcal{F}$ and $\mathcal{G}$ implies that $\mathcal{G}=\mathcal{T}_f$. Then $\tau_t(\mathcal{T}_f)=t+1$ and $\mathcal{G}$ is $t$-intersecting.
From Fact \ref{DRG} and $\tau_t(\mathcal{G})=t+1$, there exists some $M\in {[n]\choose t+2}$ such that $\mathcal{G}\subseteq {M\choose t+1}$. For each $B\in \mathcal{T}_g$ and each $G\in\mathcal{G}$, it is clear that $|B\cap M\cap G|\geq t$. Hence $|B\cap M|=t+1$ from $\tau_t(\mathcal{G})=t+1$ again. Therefore, $\mathcal{T}_g\subseteq {M\choose t+1}$ is $t$-intersecting. $\qed$

\begin{lem}\label{m}
Suppose that $\mathcal{T}_f\cup \mathcal{T}_g$ is a $t$-intersecting family with $\tau_t(\mathcal{T}_f\cup \mathcal{T}_g)=t$.
\begin{enumerate}[{\rm (i)}]
  \item $\min\{k,\ell\}\geq t+2$, $|\mathcal{T}_f|\leq k-t+1$ and $|\mathcal{T}_g|\leq \ell-t+1$.
  \item Assume that $(|\mathcal{T}_f|,|\mathcal{T}_g|)=(k-t+1,\ell-t+1)$. Then $|M_f\cap M_g|\geq t+1$, and the following statements hold.
      \begin{enumerate}
  \item[{\rm (iia)}] If $|M_f\cap M_g|= t+1$, then
  $|\mathcal{F}||\mathcal{G}|=(h(k,\ell,t)-t+1)(h(\ell, k,t)-t+1)$. Moreover, when $t=1$, we have $(\mathcal{F},\mathcal{G})\simeq (\mathcal{H}(k,1;X,Y), ~\mathcal{H}(\ell,1;Y,X))$, where $X\in {[2,n]\choose k}$ and $Y\in {[2,n]\choose \ell}$ satisfy $|X\cap Y|=1$;
  \item[{\rm (iib)}] If $|M_f\cap M_g|\geq t+2$, then
      $(\mathcal{F},\mathcal{G})\simeq (\mathcal{H}(k,t;X,Y), ~\mathcal{H}(\ell,t;Y,X))$, where $X\in {[t+1,n]\choose k-t+1}$ and $Y\in {[t+1,n]\choose \ell-t+1}$ satisfy $|X\cap Y|\geq 2$.
\end{enumerate}
\end{enumerate}
\end{lem}
\proof (i) It follows from Proposition~\ref{l+1_t+1_t} that $k\geq t+2$ and $\ell\geq t+2$.
Since $\tau_t(\mathcal{T}_f\cup \mathcal{T}_g)=t$, there exists a $t$-subset $T$ of $[n]$ such that  $\mathcal{T}_f\cup \mathcal{T}_g\subseteq [T\rangle_{t+1}$. It is clear that $\mathcal{T}_f=[T, M_f]_{t+1}$, $\mathcal{T}_g=[T,M_g]_{t+1}$, and
\begin{align}\label{TMTM}
(|\mathcal{T}_f|,|\mathcal{T}_g|)=(|M_f|-t,|M_g|-t).
\end{align}

It follows from \eqref{abs} and $\tau_t(\mathcal{F})=\tau_t(\mathcal{G})=t+1$ that all of $\mathcal{F}_T$, $\mathcal{G}_T$, $\mathcal{F}\setminus \mathcal{F}_T$ and $\mathcal{G}\setminus \mathcal{G}_T$ are nonempty.
For each $F \in \mathcal{F}\setminus \mathcal{F}_T$ and $A\in \mathcal{T}_f$, we have
$$t\geq |F\cap A|=|F\cap T|+|F\cap (A\setminus T)|\geq t,$$
implying that $|F\cap T|=t-1$ and $A\setminus T\subseteq F$. Therefore,
$$M_f=\cup_{A\in \mathcal{T}_f}A\subseteq F\cup T~\text{for~each}~F\in \mathcal{F}\setminus \mathcal{F}_T,$$
and so $|M_f|\leq k+1$. Similarly, for each $G\in \mathcal{G}\setminus \mathcal{G}_T$, we also have $|G\cap T|=t-1$ and $M_g\subseteq G\cup T$, and so $|M_g|\leq \ell+1$. By \eqref{TMTM}, (i) holds.

(ii) Observe that $(|\mathcal{T}_f|,|\mathcal{T}_g|)=(k-t+1,\ell-t+1)$, equivalent to $(|M_f|,|M_g|)= (k+1,\ell+1)$.
Then $M_f=F\cup T$ for each $F\in \mathcal{F}\setminus \mathcal{F}_T$, and $M_g=G\cup T$ for each $G\in \mathcal{G}\setminus \mathcal{G}_T$. Hence,
\begin{align}\label{F-TG-T}
\mathcal{F}\setminus \mathcal{F}_T\subseteq \big\{M_f\setminus \{x\} : x\in T\big\}
~\text{and}~
\mathcal{G}\setminus \mathcal{G}_T\subseteq \big\{ M_g\setminus \{y\} : y\in T\big\}.
\end{align}
Choose $M_f\setminus \{x\}\in \mathcal{F}\setminus \mathcal{F}_T$ and $M_g\setminus \{y\}\in \mathcal{G}\setminus \mathcal{G}_T$, where $x,y\in T$. Then
\begin{align}\label{MfcapMg=}
|M_f\cap M_g|=|(M_f\setminus \{x\})\cap (M_g\setminus \{y\})|+|\{x\}\cup \{y\}|\geq t+1.
\end{align}

Set
\begin{align*}
\mathcal{F}_0&=\left\{F\in {[n]\choose k} : T\subseteq F,~|F\cap (M_g\setminus T)|\geq 1\right\},\\
\mathcal{G}_0&=\left\{G\in {[n]\choose \ell} : T\subseteq G,~|G\cap (M_f\setminus T)|\geq 1\right\}.
\end{align*}
We can prove that
\begin{align}\label{FGF0G0+}
\mathcal{F}\subseteq \mathcal{F}_0\cup \big\{M_f\setminus \{x\} : x\in T\big\},~\mathcal{G}\subseteq \mathcal{G}_0\cup \big\{ M_g\setminus \{y\} : y\in T\big\}.
\end{align}
It is sufficient to prove that $\mathcal{F}_T\subseteq \mathcal{F}_0$ and $\mathcal{G}_T\subseteq \mathcal{G}_0$.
Choose $G\in \mathcal{G}\setminus \mathcal{G}_T$. By \eqref{F-TG-T}, there exists $y\in T$ such that $G=M_g\setminus \{y\}$. For each $F\in \mathcal{F}_T$, we have
$$|F\cap G|=|F\cap (M_g\setminus T)|+|F\cap (T\setminus \{y\})|\geq t,$$
implying that $|F\cap (M_g\setminus T)|\geq 1$.
Hence, $F\in \mathcal{F}_0$, and so $\mathcal{F}_T\subseteq\mathcal{F}_0$.
Similarly, one can prove  $\mathcal{G}_T\subseteq \mathcal{G}_0$.

(iia) Suppose $|M_f\cap M_g|=t+1$. By \eqref{FGF0G0+}, there exists an element $x\in T$ such that
$\mathcal{F}\setminus\mathcal{F}_T=\big\{M_f\setminus \{x\}\big\}~\text{and}~\mathcal{G}\setminus \mathcal{G}_T=\big\{M_g\setminus \{x\}\big\}$. Since $\mathcal{F}_0\cup \big\{M_f\setminus \{x\}\big\}$ and $\mathcal{G}_0\cup \big\{M_g\setminus \{x\}\big\}$ are cross $t$-intersecting, by the maximality of $\mathcal{F}$ and $\mathcal{G}$, we have
$$\mathcal{F}=\mathcal{F}_0\cup \big\{M_f\setminus \{x\}\big\}~\text{and}~\mathcal{G}=\mathcal{G}_0\cup \big\{M_g\setminus \{x\}\big\}.$$
Direct computation yields $|\mathcal{F}|= h(k,\ell,t)-t+1$ and $|\mathcal{G}|= h(\ell, k,t)-t+1$. In particular, when $t=1$, it follows directly that $|(M_f\setminus T)\cap (M_g\setminus T)|=1$ and
$$(\mathcal{F},\mathcal{G})= (\mathcal{H}(k,1;M_f\setminus T,M_g\setminus T), ~\mathcal{H}(\ell,1;M_g\setminus T,M_f\setminus T)).$$
Then (iia) holds.

(iib) Suppose $|M_f\cap M_g|\geq t+2$. It is routine to check that $\mathcal{F}_0\cup \big\{M_f\setminus \{x\} : x\in T\big\}$ and $\mathcal{G}_0\cup \big\{ M_g\setminus \{y\} : y\in T\big\}$ are cross $t$-intersecting. By the maximality of $\mathcal{F}$ and $\mathcal{G}$, we have $$\mathcal{F}=\mathcal{F}_0\cup \big\{M_f\setminus \{x\} : x\in T\big\}~\text{and}~\mathcal{G}=\mathcal{G}_0\cup \big\{ M_g\setminus \{y\} : y\in T\big\}.$$
It follows immediately that $|(M_f\setminus T)\cap (M_g\setminus T)|\geq 2$ and
$$(\mathcal{F},\mathcal{G})= (\mathcal{H}(k,t;M_f\setminus T,M_g\setminus T), ~\mathcal{H}(\ell,t;M_g\setminus T,M_f\setminus T)).$$
Then (iib) holds. $\qed$

Let $x$ be a positive integer and $M\subseteq [n]$. Denote
$$\mathcal{A}_x(M)= \left\{F\in {[n]\choose x} : |F\cap M|\geq t+1\right\}.$$

\begin{lem}\label{equivalent}
  Suppose that $\mathcal{T}_f\cup\mathcal{T}_g$ is a $t$-intersecting family with $\tau_t(\mathcal{T}_f\cup\mathcal{T}_g)= t+1$. Then $\max\big\{|\mathcal{T}_f|, |\mathcal{T}_g|\big\}\leq t+2$. Moreover, when $|\mathcal{T}_f|=|\mathcal{T}_g|=t+2$, we have $(\mathcal{F},\mathcal{G})\simeq(\mathcal{A}(k,t), \mathcal{A}(\ell,t))$.
\end{lem}
\proof  Since $\mathcal{T}_f\cup\mathcal{T}_g$ is a $t$-intersecting family with $\tau_t(\mathcal{T}_f\cup\mathcal{T}_g)= t+1$, by Fact~\ref{DRG}, we can choose a $(t+2)$-subset $M$ such that $\mathcal{T}_f\cup\mathcal{T}_g\subseteq {M\choose t+1}$. Then $\max\{|\mathcal{T}_f|, |\mathcal{T}_g|\}\leq t+2$.

Suppose $|\mathcal{T}_f|=t+2$. From $|{M\choose t+1}|=t+2$, it follows that $\mathcal{T}_f={M\choose t+1}$. Then $|F\cap M|\geq t+1$ for every $F\in \mathcal{F}$, and so $\mathcal{F}\subseteq\mathcal{A}_k(M)$. Similarly, we have $\mathcal{G}\subseteq\mathcal{A}_\ell(M)$ when $|\mathcal{T}_g|=t+2$.
It follows from $|M|=t+2$ and the maximality of $\mathcal{F}$ and $\mathcal{G}$, we have
$$(\mathcal{F},\mathcal{G})\simeq(\mathcal{A}(k,t), \mathcal{A}(\ell,t)),$$
and the lemma follows.$\qed$

In the following lemmas, we consider the case when $\mathcal{T}_f\cup\mathcal{T}_g$ is not $t$-intersecting. Without loss of generality, we may assume that $\mathcal{T}_g$ is not a $t$-intersecting family. Then there exist $B_1,~B_2\in \mathcal{T}_g$ such that $|B_1\cap B_2|\leq t-1$. Let $A\in \mathcal{T}_f$. From Fact \ref{s-tt} and $|A|=t+1$, we have
\begin{align*}
t-1\geq |A\cap B_1\cap B_2|&=|A\cap B_1|+|A\cap B_2|-|(A\cap B_1)\cup(A\cap B_2)|\\
&\geq 2t-|A\cap(B_1\cup B_2)|\geq t-1.
\end{align*}
Therefore,
\begin{align}\label{bbabb}
|B_1\cap B_2|=t-1,~\text{and}~B_1\cap B_2\subseteq A\subseteq B_1\cup B_2~\text{for~every}~A\in \mathcal{T}_f.
\end{align}

\begin{lem}\label{T=1}
Suppose $\mathcal{T}_g$ is not $t$-intersecting, and $|\mathcal{T}_f|=1$. Then the following statements hold.
\begin{enumerate}[{\rm(i)}]
  \item If $\mathcal{T}_f\nsubseteq \mathcal{T}_g$, then $|\mathcal{T}_g|\leq 2(\ell-t+1)$.
  \item If $\mathcal{T}_f\subseteq \mathcal{T}_g$, then $|\mathcal{T}_g|\leq (t+1)(\ell-t)+1$. Moreover, when $t\geq 2$ and $|\mathcal{T}_g|=(t+1)(\ell-t)+1$, we have $(\mathcal{F},\mathcal{G})\simeq(\mathcal{C}_2(k,t;\ell), ~\mathcal{C}_1(\ell,t))$.
\end{enumerate}
\end{lem}
\proof (i) Suppose $\mathcal{T}_f=\{M_f\}\nsubseteq \mathcal{T}_g$. Then there exists $G'\in \mathcal{G}$ such that $|G'\cap M_f|\leq t-1$. Let $B\in \mathcal{T}_g$. Since $|G'\cap B|\geq t$ and $|B\cap M_f|\geq t$, we have
\begin{align*}
t-1\geq |G'\cap M_f|\geq |G'\cap M_f\cap B|=|G'\cap B|+|M_f\cap B|-|(G'\cup M_f)\cap B|\geq t-1.
\end{align*}
Hence,
$$|G'\cap M_f|=t-1,~G'\cap M_f\subseteq B~\text{and}~|G'\cap B|=|M_f\cap B|=t.$$
Then we may write $B=(G'\cap M_f)\cup \{x_B, y_B\}$, where $x_B\in M_f\setminus G'$ and $y_B\in G'\setminus M_f$. From $|M_f\setminus G'|=2$ and $|G'\setminus M_f|=\ell-t+1$, it follows that $|\mathcal{T}_g|\leq 2(\ell-t+1)$.

(ii) Since $\mathcal{T}_f=\{M_f\}\subseteq \mathcal{T}_g$, by Fact \ref{s-tt}, we have $|B\cap M_f|=t$ for each $B\in \mathcal{T}_g\setminus \{M_f\}$, and thus
\begin{align*}
\mathcal{T}_g\setminus \{M_f\}=\cup_{T\in {M_f\choose t}}\big((\mathcal{T}_g)_T\setminus \{M_f\}\big).
\end{align*}
Let $T\in  {M_f\choose t}$, and set
$$N(T)=\big\{x\in [n]\setminus M_f : T\cup \{x\}\in (\mathcal{T}_g)_T\setminus \{M_f\}\big\}.$$
Then
\begin{align}\label{TgNT}
\mathcal{T}_g=\cup_{T\in {M_f\choose t}}\big\{T\cup \{x\} : x\in N(T)\big\}\bigcup \big\{M_f\big\}.
\end{align}
Since $M_f\in \mathcal{T}_g$ and $\tau_t(\mathcal{G})=t+1$, there exist $G_1,G_2\in \mathcal{G}\setminus \mathcal{G}_{M_f}$ such that $G_1\cap M_f\neq G_2\cap M_f$.
Let $G_T\in \{G_1,G_2\}$ with $T\neq G_T\cap M_f$. Observe that $|G_T\cap M_f|=t$ and $|G_T\cap T|<t$.
For each $x\in N(T)$, since $|G_T\cap (T\cup \{x\})|\geq t$, we have $x\in G_T\setminus M_f$. It follows that $N(T)\subseteq G_T\setminus M_f$, and so $|N(T)|\leq \ell-t$. Therefore,
$$|\mathcal{T}_g|=\sum_{T\in {M_f\choose t}}|(\mathcal{T}_g)_T\setminus \{M_f\}|+1=\sum_{T\in {M_f\choose t}}|N(T)|+1\leq (t+1)(\ell-t)+1.$$

Suppose that $t\geq 2$ and $|\mathcal{T}_g|=(t+1)(\ell-t)+1$. Then for every $T\in {M_f\choose t}$, we have $|N(T)|=\ell-t$, and so $N(T)=G_T\setminus M_f$ for each $G_T\in \{G_1,G_2\}$ with $|G_T\cap T|<t$. By $t\geq 2$, it is clear that $|{M_f\choose t}|\geq 3$, and there exists $T_0\in {M_f\choose t}$ such that $|T_0\cap G_1|<t$ and $|T_0\cap G_2|<t$. Thus $N(T_0)=G_1\setminus M_f=G_2\setminus M_f$, implying that $N(T)=G_1\setminus M_f$ for all $T\in {M_f\choose t}$.
Observe that $N(T)=\cup_{T'\in {M_f\choose t}}N(T')=M_g\setminus M_f$. Hence, by \eqref{TgNT}, we have
\begin{align}\label{Tg=TCx}
\mathcal{T}_g=\left\{T\cup \{x\} : x\in M_g\setminus M_f,~ T\in {M_f\choose t}\right\}\bigcup \big\{M_f\big\}.
\end{align}

Set
\begin{align*}
\mathcal{F}'&=\left\{F\in {[n]\choose k} : |F\cap M_f|=t,~|F\cap (M_g\setminus M_f)|\geq 1\right\}\bigcup [M_f\rangle_k,\\
\mathcal{G}'&=\left\{T\cup (M_g\setminus M_f) : T\in {M_f\choose t}\right\}\bigcup \big[M_f\big\rangle_\ell.
\end{align*}
Note that $\mathcal{F}'$ and $\mathcal{G}'$ are cross $t$-intersecting.
By \eqref{abs}, we see that $[M_f\rangle_\ell\subseteq \mathcal{G}$. Let $G\in \mathcal{G}\setminus [M_f\rangle_\ell$, then there exists $T'\in {M_f\choose t}$ such that $G\cap M_f\neq T'$.
For each $x\in M_g\setminus M_f$, since $T'\cup \{x\}\in \mathcal{T}_g$, we have $$|G\cap (T'\cup \{x\})|=|G\cap T'|+|G\cap \{x\}|\geq t,$$
implying that $x\in G$. Hence, $G=(G\cap M_f)\cup(M_g\setminus M_f)\in \mathcal{G}'.$
Therefore,
$\mathcal{G}\subseteq \mathcal{G}'$.
Let $F\in \mathcal{F}$. If $M_f\subseteq F$, then $F\in [M_f\rangle_k\subseteq \mathcal{F}'$. Suppose $M_f\nsubseteq F$. Then $|F\cap M_f|=t$ from $\mathcal{T}_f=\{M_f\}$. By $\tau_t(\mathcal{G})=t+1$, there exists $G\in \mathcal{G}\setminus [M_f\rangle_\ell$ with $F\cap M_f\neq G\cap M_f$. Then $|F\cap (G\cap M_f)|=|(F\cap M_f)\cap (G\cap M_f)|\leq t-1$, and $G=(G\cap M_f) \cup (M_g\setminus M_f)$ from $\mathcal{G}\subseteq \mathcal{G}'$. Since
$$|F\cap G|=|F\cap (G\cap M_f)|+|F\cap (M_g\setminus M_f)|\geq t,$$
we obtain $|F\cap (M_g\setminus M_f)|\geq 1$, implying that $F\in \mathcal{F}'$. Hence, $\mathcal{F}\subseteq \mathcal{F}'$.
From the maximality of $\mathcal{F}$ and $\mathcal{G}$, it follows that $(\mathcal{F},\mathcal{G}) =(\mathcal{F}', \mathcal{G}')$.
Furthermore, it is routine to check that $(\mathcal{F},~\mathcal{G})\simeq (\mathcal{C}_2(k,t;\ell),~\mathcal{C}_1(\ell,t))$, and the latter part of (ii) holds.$\qed$

\begin{lem}\label{T=2-C}
 Suppose that $\mathcal{T}_g$ is not $t$-intersecting, and $\mathcal{T}_f$ is a $t$-intersecting family with $|\mathcal{T}_f|\geq 2$. Then the following statements hold.
\begin{enumerate}[{\rm(i)}]
  \item $|\mathcal{T}_f|=2$ and $|\mathcal{T}_g|\leq \ell+1$.
  \item If $|\mathcal{T}_g|=\ell+1$, then $|\mathcal{F}||\mathcal{G}|\leq(\ell+1) {n-t-1\choose k-t-1}h(\ell,t+1,t).$ Moreover, equality holds if and only if $k=t+1$ and there exist $X\in{[t+1,n]\choose 2}$ and $Y\in{[t+1,n]\choose \ell-t+1}$ with $|X\cap Y|\geq  2-\delta_{1,t}$ such that $(\mathcal{F},\mathcal{G})\simeq \left(\mathcal{H}(t+1,t;X,Y), ~\mathcal{H}(\ell,t;Y,X) \right)$.
\end{enumerate}
\end{lem}
\proof (i)
Let $B_1$ and $B_2$ be distinct elements in $\mathcal{T}_g$ with $|B_1\cap B_2|\leq t-1$. It follows from Fact \ref{s-tt} that $B_1,B_2\notin \mathcal{T}_f$. From \eqref{bbabb}, we know that $|B_1\cup B_2|=|B_1|+|B_2|-|B_1\cap B_2|=t+3$ and $M_f=\cup_{A\in \mathcal{T}_f}A\subseteq B_1\cup B_2\subseteq M_g$.

Assume that $|\mathcal{T}_f|\geq 3$. Let $A_1$, $A_2$ and $A_3$ be distinct elements in $\mathcal{T}_f$. For each $i\in [2]$, we have $\mathcal{E}_i=\{A_1,A_2,A_3,B_i\}\subseteq {B_1\cup B_2\choose t+1}$ is a $t$-intersecting family from Fact \ref{s-tt}. If $\tau_t(\mathcal{E}_i)=t+1$ for each $i\in [2]$, from Fact \ref{DRG} and $|A_1\cup A_2|=t+2$, then $B_i\in\mathcal{E}_i\subseteq {A_1\cup A_2\choose t+1}$, contradicting $|B_1\cap B_2|\leq t-1$. If there exists $j\in [2]$ such that $\tau_t(\mathcal{E}_j)=t$, then $\mathcal{E}_j\subseteq [A_1\cap A_2,B_1\cup B_2]_{t+1}$, and so  $4=|\mathcal{E}_j|\leq \big|[A_1\cap A_2,B_1\cup B_2]_{t+1}\big|=3$, a contradiction. Hence $|\mathcal{T}_f|=2$.

Let $\mathcal{T}_f=\{A_1,A_2\}$ and $T=A_1\cap A_2$. For each $B\in \mathcal{T}_g$, since $|B\cap A_1|\geq t$ and $|B\cap A_2|\geq t$, it follows from Fact \ref{DRG} that either $T\subseteq B$ or $B\subseteq M_f$. Then
\begin{align}\label{Tsub}
\mathcal{T}_g\subseteq [T,M_g]_{t+1}\cup {M_f\choose t+1},
\end{align}
and so
\begin{align}\label{Tsub-1}
|\mathcal{T}_g|\leq \biggl|[T,M_g]_{t+1}\cup {M_f\choose t+1}\biggl|=|M_g|,
\end{align}
from $[T,M_g]_{t+1}\cap {M_f\choose t+1}=\mathcal{T}_f$.

Let $y\in M_g\setminus M_f$. There exists $B_y\in \mathcal{T}_g$ such that $y\in B_y$. From $B_y\nsubseteq M_f$, we have $B_y=T\cup \{y\}$, and so $|G\cap (T\cup\{y\})|\geq t$ for all $G\in \mathcal{G}$. Hence,
\begin{align}\label{Mg-Mf<G}
\text{for~each}~G\in\mathcal{G}~\text{with}~|G\cap T|\leq t-1,~\text{we~have}~y\in G,~\text{and~so}~M_g\setminus M_f\subseteq G.
\end{align}
From $\tau_t(\mathcal{G})=t+1$, there exists $G'\in \mathcal{G}$ such that $|G'\cap T|\leq t-1$.
Then by \eqref{Mg-Mf<G}, we know
\begin{align}\label{Mg<l+1}
|M_g|\leq |G'\cup M_f| = \ell+t+2-|G'\cap M_f|.
\end{align}
Since $\mathcal{T}_g$ is not $t$-intersecting, by \eqref{Tsub}, we have $\mathcal{T}_g\cap {M_f\choose t+1}\neq \emptyset$, and so $|G'\cap M_f|\geq t$. If $|G'\cap M_f|\geq t+1$, then $|\mathcal{T}_g|\leq |M_g|\leq \ell+1$. Now suppose $|G'\cap M_f|=t$. Since $|{M_f\choose t+1}|\geq 3$, there exists $B\in {M_f\choose t+1}$ such that $|G'\cap B|<t$. Hence, $B\notin \mathcal{T}_g$, implying that $|\mathcal{T}_g|<|M_g|\leq \ell+2$ due to \eqref{Tsub-1} and \eqref{Mg<l+1}. Therefore, $|\mathcal{T}_g|\leq \ell+1$ holds.

(ii) Suppose $|\mathcal{T}_g|=\ell+1$.  By \eqref{Tsub-1}, we have $|M_g|\in \{\ell+1,\ell+2\}$ and $|\{A_1,A_2\}\cap \mathcal{T}_g|\geq 1$.
Without loss of generality, assume that $A_1\in \mathcal{T}_g$.
Since $\mathcal{T}_g$ is not $t$-intersecting, there exists $B\in \mathcal{T}_g$ such that  $B\in {M_f\choose t+1}$ and $T\nsubseteq B$.
Let $G\in \mathcal{G}\setminus ([A_1\rangle_\ell\cup[A_2\rangle_\ell)$. Thus,
$|G\cap T|\geq |G\cap A_1|-1\geq t-1.$
If $T\subseteq G$, then
$$|G\cap M_f|=|G\cap (A_1\cup B)|=|G\cap A_1|+|G\cap B|-|G\cap A_1\cap B|\geq 2t-|T\cap B|\geq t+1,$$
implying that $A_1\subseteq G$ or $A_2\subseteq G$, a contradiction. Hence,
\begin{align}\label{GCT=t-1}
|G\cap T|=t-1~\text{for~all}~G\in \mathcal{G}\setminus ([A_1\rangle_\ell\cup[A_2\rangle_\ell).
\end{align}

Next, we will discuss the two cases $|M_g|=\ell+1$ and $|M_g|=\ell+2$ separately.

\textbf{\textrm{Case~1.}} $|M_g|=\ell+1$.

Then $|\mathcal{T}_g|= |M_g|$, and from \eqref{Tsub}, we have $\mathcal{T}_g=[T,M_g]_{t+1}\cup {M_f\choose t+1}.$

Set
\begin{align*}
\mathcal{F}'&=\left\{F\in {[n]\choose k} : T\subseteq F,~|F\cap (M_g\setminus T)|\geq 1\right\}\bigcup \mathcal{A}_k(M_f),\\
\mathcal{G}'&= [A_1\rangle_\ell\bigcup[A_2\rangle_\ell \bigcup \big\{M_g\setminus\{x\} : x\in T\big\} .
\end{align*}
Note that $\mathcal{F}'$ and $\mathcal{G}'$ are cross $t$-intersecting.
Let $G\in \mathcal{G}\setminus ([A_1\rangle_\ell\cup[A_2\rangle_\ell)$.
Since ${M_f\choose  t+1}\subseteq \mathcal{T}_g$, we have $|G\cap M_f|\geq t+1$.
By \eqref{GCT=t-1}, we know $|G\cap T|=t-1$, and so $M_g\setminus M_f\subseteq G$ from \eqref{Mg-Mf<G}. Thus, $$\ell\geq|G\cap M_g|=|G\cap M_f|+|G\cap (M_g\setminus M_f)|\geq \ell,$$
implying that $G\subseteq M_g$, and so $G=M_g\setminus \{x\}$ for some $x\in T$. Hence, $\mathcal{G}\subseteq \mathcal{G}'$.
Let $F\in \mathcal{F}\setminus \mathcal{A}_k(M_f)$. Then $|F\cap M_f|=t$. It follows from $|F\cap A_1|\geq t$ and $|F\cap A_2|\geq t$ that $F\cap M_f=T$. Choose $G=M_g\setminus \{x\}$ for some $x\in T$. Since
$$|F\cap G|=|F\cap(M_g\setminus \{x\})|=|F\cap (M_g\setminus T)|+|F\cap (T\setminus \{x\})|\geq t,$$
and $|F\cap (T\setminus \{x\})|=|T\setminus \{x\}|=t-1$, we have $|F\cap (M_g\setminus T)|\geq 1$, implying that $F\in \mathcal{F}'$. Hence, $\mathcal{F}\subseteq \mathcal{F}'$.

Therefore, by the maximality of $\mathcal{F}$ and $\mathcal{G}$, we have $\mathcal{F}=\mathcal{F}'$ and $\mathcal{G}=\mathcal{G}'$.
When $k\geq t+2$, it is routine to check that $\mathcal{F}\subseteq \cup_{B\in \mathcal{T}_g}[B\rangle_k$, and so $|\mathcal{F}||\mathcal{G}|< (\ell+1) {n-t-1\choose k-t-1}h(\ell,t+1,t)$.
When $k=t+1$, it is straightforward to verify that $M_f\setminus T\subseteq M_g\setminus T$ and
 $$(\mathcal{F},~\mathcal{G})= \left(\mathcal{H}(t+1,t;M_f\setminus T,M_g\setminus T), ~\mathcal{H}(\ell,t;M_g\setminus T,M_f\setminus T) \right).$$
Therefore, (ii) holds when $|M_g|=\ell+1$.

\textbf{\textrm{Case~2.}} $|M_g|=\ell+2$.

Since $|[T,M_g]_{t+1}\setminus \mathcal{T}_f|=\ell-t$, it follows from \eqref{Tsub} that
$$\left|\mathcal{T}_g\cap {M_f\choose t+1}\right|\geq |\mathcal{T}_g|-\left|[T,M_g]_{t+1}\setminus \mathcal{T}_f\right|=t+1\geq 2.$$
If $|\mathcal{T}_g\cap {M_f\choose t+1}|\geq 3$, then by Fact \ref{DRG}, we have $|G\cap M_f|\geq t+1$ for all $G\in \mathcal{G}$, and so $|M_g|\leq \ell+1$ from \eqref{Mg<l+1}, a contradiction.
Hence $\left|\mathcal{T}_g\cap {M_f\choose t+1}\right|=2$, and so $t=1$.
Let ${M_f\choose 2}=\{A_1,A_2,B_0\}$. Recall our assumption that $A_1\in \mathcal{T}_g$. Since $\mathcal{T}_g$ is not $t$-intersecting, then $B_0\in \mathcal{T}_g$, and so $\mathcal{T}_g\cap {M_f\choose t+1}=\{A_1,B_0\}$.

Set
\begin{align*}
\mathcal{F}''&=\left\{F\in {[n]\choose k} : T\subseteq F,~|F\cap (M_g\setminus A_2)|\geq 1\right\}\bigcup [B_0\rangle_k,\\
\mathcal{G}''&= [A_1\rangle_\ell\bigcup[A_2\rangle_\ell \bigcup \big\{M_g\setminus A_2\big\}.
\end{align*}
Note that $\mathcal{F}''$ and $\mathcal{G}''$ are cross $t$-intersecting. Recall that $M_f\subseteq M_g$.
Let $G\in \mathcal{G}\setminus ([A_1\rangle_\ell\cup[A_2\rangle_\ell)$.
Then $M_g\setminus M_f\subseteq G$ by \eqref{Mg-Mf<G}, and $G\cap T=\emptyset$ by \eqref{GCT=t-1}. Since $|M_g\setminus M_f|=\ell-1$, we have $|G\cap M_f|=|G\cap A_1|=|G\cap B_0|=1$, and so $G\cap M_f=A_1\cap B_0=M_f\setminus A_2$ by $M_f=(A_1\cap B_0)\cup A_2$.
Thus, $(M_g\setminus M_f)\cup (M_f\setminus A_2)=M_g\setminus A_2\subseteq G$, and so $G=M_g\setminus A_2$. Hence, $\mathcal{G}\subseteq \mathcal{G}''$.
Let $F\in \mathcal{F}$. It is clear that $|F\cap G|=|F\cap (M_g\setminus A_2)|\geq 1$.
If $T\subseteq F$, then $F\in \mathcal{F}''$.
If $T\nsubseteq F$, then from $|F\cap A_1|\geq 1$, $|F\cap A_2|\geq 1$ and ${M_f\choose 2}=\{A_1,A_2,B_0\}$, we have $B_0\subseteq F$, implying $F\in [B_0\rangle_k$. Hence, $\mathcal{F}\subseteq \mathcal{F}''$.

Therefore, by the maximality of $\mathcal{F}$ and $\mathcal{G}$, we obtain $\mathcal{F}=\mathcal{F}''$ and $\mathcal{G}=\mathcal{G}''$.
When $k\geq 3$, a direct computation yields $|\mathcal{F}||\mathcal{G}|<(\ell+1) {n-2\choose k-2}h(\ell,2,1).$
When $k=2$, it is straightforward to verify that $|B_0\cap (M_g\setminus A_2)|=1$ and $$(\mathcal{F},~\mathcal{G})=\left(\mathcal{H}(2,1; B_0,M_g\setminus A_2), ~\mathcal{H}(\ell,1;M_g\setminus A_2,B_0)\right).$$
Then (ii) holds when $|M_g|=\ell+2$. This completes the proof.$\qed$

\begin{lem}\label{both-non-inter}
Suppose neither $\mathcal{T}_f$ nor $\mathcal{T}_g$ is $t$-intersecting. The following statements hold.
\begin{enumerate}[{\rm(i)}]
  \item $|\mathcal{T}_f|, |\mathcal{T}_g|\in \{2,3\}$, or $(|\mathcal{T}_f|, |\mathcal{T}_g|)\in \big\{(2,4),~(4,2)\big\}$.
  \item If $|\mathcal{T}_f|=|\mathcal{T}_g|=3$, then $|\mathcal{F}||\mathcal{G}|= \left(a(k,t)-(t-1){n-t-3\choose k-t-1}\right)\left(a(\ell,t)-(t-1){n-t-3\choose \ell-t-1}\right).$ Moreover, when $t=1$, we have
  $(\mathcal{F},~\mathcal{G})\simeq (\mathcal{B}(k;(1,3,2,4)), ~\mathcal{B}(\ell;(1,2,3,4)))$.
\end{enumerate}
\end{lem}
\proof
(i) Since neither $\mathcal{T}_f$ nor $\mathcal{T}_g$ is $t$-intersecting, we have $|\mathcal{T}_f|\geq 2$ and $|\mathcal{T}_g|\geq 2$.
Let $A_1$ and $A_2$ be distinct elements in $\mathcal{T}_f$ with $|A_1\cap A_2|\leq t-1$, and $B_1$ and $B_2$ be distinct elements in $\mathcal{T}_g$ with $|B_1\cap B_2|\leq t-1$.
By \eqref{bbabb}, we have $|A_1\cap A_2|=|B_1\cap B_2|=t-1$, and $A_1\cap A_2\subseteq B\subseteq A_1\cup A_2$ and $B_1\cap B_2\subseteq A \subseteq B_1\cup B_2$ for each $A\in \mathcal{T}_f$ and $B\in \mathcal{T}_g$, implying that $A_1\cap A_2=B_1\cap B_2$ and $M_g=A_1\cup A_2=B_1\cup B_2=M_f\in {[n]\choose t+3}$.

Let $A_1\cap A_2=T$. Suppose $A_1=T\cup \{x_1,x_2\}$, $A_2=T\cup \{x_3,x_4\}$, where $x_1$, $x_2$, $x_3$ and $x_4$ are the four elements of $M_f\setminus T$. It follows from Fact~\ref{s-tt} that $B_1,B_2\notin \mathcal{T}_f$ and $A_1,A_2\notin \mathcal{T}_g$. So we can let $B_1=T\cup \{x_1,x_3\}$ and $B_2=T\cup \{x_2,x_4\}$, and set $C_1=T\cup\{x_2,x_3\}$ and $C_2=T\cup \{x_1,x_4\}$.
Hence,
$$\{A_1,A_2\}\subseteq \mathcal{T}_f\subseteq \{A_1,A_2,C_1,C_2\}~\text{and}~\{B_1,B_2\}\subseteq \mathcal{T}_g\subseteq \{B_1,B_2,C_1,C_2\}.$$
If $|\mathcal{T}_f|=4$, then $\mathcal{T}_f= \{A_1,A_2,C_1,C_2\}$, and so $\mathcal{T}_g=\{B_1,B_2\}$. Therefore, by symmetry, we can obtain that
$|\mathcal{T}_f|, |\mathcal{T}_g|\in\{2,3\}~\text{or}~(|\mathcal{T}_f|, |\mathcal{T}_g|)\in \big\{(2,4),~(4,2)\big\},$
and (i) holds.

(ii) Suppose $|\mathcal{T}_f| = |\mathcal{T}_g| =3$. Without loss of generality, we assume $\mathcal{T}_f=\{A_1,A_2,C_1\}$. Then from Fact \ref{s-tt}, we have $\mathcal{T}_g = \{B_1,B_2,C_1\}$. Observe that $M_f=M_g=T\cup \{x_1,x_2,x_3,x_4\}$. Denote
\begin{align*}
\mathcal{F}_0&=[B_1\rangle_k\cup [B_2\rangle_k \cup [C_1\rangle_k\cup\left\{F\in {[n]\choose k} : |F\cap M_f|\geq t+2 \right\},\\
\mathcal{G}_0&=[A_1\rangle_\ell \cup [A_2\rangle_\ell \cup [C_1\rangle_\ell\cup\left\{G\in {[n]\choose \ell} : |G\cap M_g|\geq t+2 \right\}.
\end{align*}

Let $F\in \mathcal{F}$. If $|F\cap M_f|\geq t+2$ or $C_1\subseteq F$, then $F\in \mathcal{F}_0$. Now suppose $|F\cap M_f|\leq t+1$ and $C_1\nsubseteq F$. If $T\nsubseteq F$, then $t\geq 2$ and it follows from $A_1,A_2\in \mathcal{T}_f$ that $|F\cap T|=t-2$, $\{x_1,x_2\}\subseteq F$ and $\{x_3,x_4\}\subseteq F$, which contradicts $|F\cap M_f|\leq t+1$. Hence $T\subseteq F$ and $|F\cap \{x_2,x_3\}|=1$. From $A_1,A_2\in \mathcal{T}_f$ again, we have $|F\cap \{x_1,x_2\}|\geq 1$ and $|F\cap \{x_3,x_4\}|\geq 1$, implying that $\{x_1,x_3\}\subseteq F$ or $\{x_2,x_4\}\subseteq F$. Thus $B_1\subseteq F$ or $B_2\subseteq F$, and so $F\in \mathcal{F}_0$. Therefore, $\mathcal{F}\subseteq \mathcal{F}_0$.
Similarly, one can prove $\mathcal{G}\subseteq \mathcal{G}_0$. It is clear that $\mathcal{F}_0$ and $\mathcal{G}_0$ are cross $t$-intersecting. By the maximality of $\mathcal{F}$ and $\mathcal{G}$, we have $\mathcal{F}=\mathcal{F}_0$ and $\mathcal{G}=\mathcal{G}_0$.
A direct computation yields
\begin{align*}
|\mathcal{F}|&=3{n-t-3\choose k-t-1}+(t+3){n-t-3\choose k-t-2}+{n-t-3\choose k-t-3}=a(k,t)-(t-1){n-t-3\choose k-t-1},\\
|\mathcal{G}|&=3{n-t-3\choose \ell-t-1}+(t+3){n-t-3\choose \ell-t-2}+{n-t-3\choose \ell-t-3}=a(\ell,t)-(t-1){n-t-3\choose \ell-t-1}.
\end{align*}
Thus the former part of (ii) holds.

Let $t=1$. Then $\mathcal{F}=\big[\{x_1,x_3\}\big\rangle_k \cup\big[\{x_2,x_4\}\big\rangle_k \cup\big[\{x_2,x_3\}\big\rangle_k$, implying that $\mathcal{F}\simeq \mathcal{B}(k;(1,3,2,4))$. By symmetry, $\mathcal{G}\simeq \mathcal{B}(\ell;(1,2,3,4))$ is obvious. Hence, the latter part of (ii) holds, and the lemma follows.$\qed$

\section{The proofs of the main theorems}\label{2.2}

Before proving Theorem \ref{main-1}, we will give an upper bound of the product $|\mathcal{F}||\mathcal{G}|$ with
$(\tau_t(\mathcal{F}),\tau_t(\mathcal{G}))\neq (t+1,t+1)$.

\begin{lem}\label{cfcg-t+2}
Let $n$, $k$, $\ell$ and $t$ be positive integers satisfying $\min\{k, \ell\}\geq t+1$ and $n\geq (t+1)^2(k+\ell)^2(k-t+1)(\ell-t+1)+k+\ell-t$.
Suppose $\mathcal{F} \subseteq \binom{[n]}{k}$ and $\mathcal{G} \subseteq \binom{[n]}{\ell}$ are maximal cross $t$-intersecting families with
$(\tau_t(\mathcal{F}),\tau_t(\mathcal{G}))\neq (t+1,t+1)$. Then
$|\mathcal{F}||\mathcal{G}|<a(k,t)a(\ell,t).$
\end{lem}
\proof
Applying \cite[Lemma~2.9]{Cao-Lu-Lv-Wang-2024} to $\mathcal{F}$ and $\mathcal{G}$ respectively, we have
\begin{align}\label{s-upper-F1}
	|\mathcal{F}|\leq\left\{
	\begin{array}{ll}
		(\ell-t+1){\tau_t(\mathcal{F})\choose t}{n-t-1\choose k-t-1}, &\text{if}~\tau_t(\mathcal{G})=t+1,\vspace{0.1cm}\\
		{\ell}^{\tau_t(\mathcal{G})-t-2}(\ell-t+1)^2 {\tau_t(\mathcal{F})\choose t}{n-\tau_t(\mathcal{G})\choose k-\tau_t(\mathcal{G})},&\text{if}~\tau_t(\mathcal{G})\geq t+2,
	\end{array}
	\right.
\end{align}
and
\begin{align}\label{s-upper-F2}
	|\mathcal{G}|\leq\left\{
	\begin{array}{ll}
		(k-t+1){\tau_t(\mathcal{G})\choose t}{n-t-1\choose
        \ell-t-1},&\text{if}~\tau_t(\mathcal{F})=t+1,\vspace{0.1cm}\\
		{k}^{\tau_t(\mathcal{F})-t-2}(k-t+1)^2
       {\tau_t(\mathcal{G})\choose t} {n-\tau_t(\mathcal{F})\choose\ell-\tau_t(\mathcal{F})},
       &\text{if}~\tau_t(\mathcal{F})\geq t+2.
	\end{array}
	\right.
\end{align}

Suppose $\tau_t(\mathcal{G})\geq t+2$  and $\tau_t(\mathcal{F})\geq t+2$. From (\ref{s-upper-F1}), (\ref{s-upper-F2}) and Lemma~\ref{notknow}~(ii), we have
\begin{align*}
|\mathcal{F}||\mathcal{G}|
&\leq (t+2)^4(k-t+1)^2(\ell-t+1)^2{n-t-2\choose k-t-2}{n-t-2\choose \ell-t-2}\\
&\leq (t+2)^4\cdot \frac{(k-t+1)^3(\ell-t+1)^3}{(n-k)(n-\ell)}{n-t-2\choose k-t-1}{n-t-2\choose \ell-t-1}\\
&\leq  \frac{(t+2)^2}{(t+1)^4(k+\ell)^2}\cdot(t+2)^2{n-t-2\choose k-t-1}{n-t-2\choose \ell-t-1}<a(k,t)a(\ell,t).
\end{align*}

Suppose $\tau_t(\mathcal{G})\geq t+2$ and $\tau_t(\mathcal{F})=t+1$. From (\ref{s-upper-F1}), (\ref{s-upper-F2}) and Lemma~\ref{notknow}~(ii), we have
\begin{align*}
|\mathcal{F}||\mathcal{G}|&\leq (t+2)^3(k-t+1)(\ell-t+1)^2{n-t-2\choose k-t-2}{n-t-1\choose \ell-t-1}
\\
&\leq (t+2)^3\cdot\frac{(k-t+1)^2(\ell-t+1)^2(n-t-1)}{(n-k)(n-\ell)} {n-t-2\choose k-t-1}{n-t-2\choose \ell-t-1}
\\
&\leq \frac{(t+2)(n-t-1)}{(t+1)^2(n-\ell)}\cdot(t+2)^2{n-t-2\choose k-t-1}{n-t-2\choose \ell-t-1}<a(k,t)a(\ell,t).
\end{align*}
By symmetry, if $\tau_t(\mathcal{F})\geq t+2$ and $\tau_t(\mathcal{G})=t+1$, then $|\mathcal{F}||\mathcal{G}|<a(k,t)a(\ell,t)$ holds, and so this lemma follows.$\qed$

\textbf{\emph{The proof of Theorem \ref{main-1}.}}
Let $\mathcal{F}_1\subseteq {[n]\choose k_1}$ and $\mathcal{F}_2\subseteq {[n]\choose k_2}$ be cross $t$-intersecting families with $\tau_t(\mathcal{F}_1) \geq t+1$ and $\tau_t(\mathcal{F}_2) \geq t+1$ such that $|\mathcal{F}_1||\mathcal{F}_2|$ is maximized.
By Lemma~\ref{cfcg-t+2}, if $(\tau_t(\mathcal{F}_1),\tau_t(\mathcal{F}_2))\neq (t+1,t+1)$, then
$|\mathcal{F}_1||\mathcal{F}_2|<a(k_1,t)a(k_2,t),$ a contradiction.
Hence $\tau_t(\mathcal{F}_1)=\tau_t(\mathcal{F}_2)=t+1$.
Let $\mathcal{T}_1$ and $\mathcal{T}_2$ be the collections of all $t$-covers of $\mathcal{F}_1$ and $\mathcal{F}_2$ with size $t+1$, respectively.

\textbf{Case~1.} Both $\mathcal{T}_f$ and $\mathcal{T}_g$ are $t$-intersecting.

Then $\mathcal{T}_f\cup \mathcal{T}_g$ is $t$-intersecting by Fact \ref{s-tt}.
Suppose $\tau_t(\mathcal{T}_1\cup\mathcal{T}_2)=t$.
By Lemma~\ref{m}~(i), we know that $k_1\geq k_2\geq t+2$, $|\mathcal{T}_1|\leq k_1-t+1$ and $|\mathcal{T}_2|\leq k_2-t+1$.
Suppose $(|\mathcal{T}_1|,|\mathcal{T}_2|)\neq(k_1-t+1,k_2-t+1)$.
It follows from Lemmas~\ref{not-contain}~and~\ref{<hh} (ii)
that
$|\mathcal{F}_1||\mathcal{F}_2|<h(k_1,k_2,t)h(k_2,k_1,t)$, a contradiction.
Hence $(|\mathcal{T}_1|,|\mathcal{T}_2|)=(k_1-t+1,k_2-t+1)$. Then, by Lemma~\ref{m}~(ii), we have
$$(\mathcal{F}_1,\mathcal{F}_2)\simeq (\mathcal{H}(k_1,t;X,Y),\mathcal{H}(k_2,t;Y,X)),$$
where $X\in{[t+1,n]\choose k_1-t+1}$ and $Y\in{[t+1,n] \choose k_2-t+1}$ with $|X\cap Y|\geq 2-\delta_{1,t}$.

Suppose $\tau_t(\mathcal{T}_1\cup\mathcal{T}_2)=t+1$.
By Lemma~\ref{equivalent}, we know that $\max\{|\mathcal{T}_1|,|\mathcal{T}_2|\}\leq t+2$.
If $(|\mathcal{T}_1|,|\mathcal{T}_2|)\neq (t+2,t+2)$, from Lemmas~\ref{not-contain}~and~\ref{<aa}~(i), then $|\mathcal{F}_1||\mathcal{F}_2|<a(k_1,t)a(k_2,t)$, a contradiction.
Hence $|\mathcal{T}_1|=|\mathcal{T}_2|=t+2$. By Lemma~\ref{equivalent}, we have
$$(\mathcal{F}_1,\mathcal{F}_2)\simeq(\mathcal{A}(k_1,t), \mathcal{A}(k_2,t)).$$

\textbf{Case~2.} For $(i,j)\in \{(1,2),(2,1)\}$, $\mathcal{T}_i$ is not $t$-intersecting, and $|\mathcal{T}_j|=1$.

By Proposition \ref{l+1_t+1_t}, we know $k_i\geq t+2$.
Suppose $\mathcal{T}_j\nsubseteq \mathcal{T}_i$. By Lemma~\ref{T=1}~(i), we have $|\mathcal{T}_i|\leq 2(k_i-t+1)$. It follows from Lemmas~\ref{not-contain} and \ref{<hh}~(i) that $|\mathcal{F}_1||\mathcal{F}_2|<h(k_1,k_2,t)h(k_2,k_1,t)$, which yields a contradiction.
Hence, $\mathcal{T}_j\subseteq \mathcal{T}_i$. Consequently, Lemma~\ref{T=1}~(ii) implies that $|\mathcal{T}_i|\leq (t+1)(k_i-t)+1$. If $|\mathcal{T}_i|\leq (t+1)(k_i-t)$, then Lemmas~\ref{not-contain} and \ref{<cc}~(i) give  $|\mathcal{F}_1||\mathcal{F}_2|<c_1(k_i,t)c_2(k_j,k_i,t)$.
If $t=1$ and $|\mathcal{T}_i|=(t+1)(k_i-t)+1$, Lemmas~\ref{not-contain} and \ref{<hh}~(i) further implies $|\mathcal{F}_1||\mathcal{F}_2|<h(k_1,k_2,1)h(k_2,k_1,1)$. In both cases, we obtain a contradiction.

Therefore, $t\geq 2$ and $|\mathcal{T}_i|=(t+1)(k_i-t)+1$.
Then, by Lemma~\ref{T=1}~(ii), $(\mathcal{F}_i,\mathcal{F}_j)\simeq(\mathcal{C}_1(k_i,t), ~\mathcal{C}_2(k_j,t;k_i))$.
If $(i,j)=(2,1)$, then $|\mathcal{F}_1||\mathcal{F}_2|=c_1(k_2,t)c_2(k_1,k_2,t)$, and from Lemma~\ref{cc<hac}, we have
\begin{align*}
	c_1(k_2,t)c_2(k_1,k_2,t)
<\left\{
	\begin{array}{ll}
        h(k_1,k_2,t)h(k_2,k_1,t),&\text{if}~k_1=k_2\geq 2t, \vspace{0.1cm}\\
        a(k_1,t)a(k_2,t),&\text{if}~t+2\leq k_1=k_2\leq 2t-1,\vspace{0.1cm}\\
        c_1(k_1,t)c_2(k_2,k_1,t),&\text{if}~k_1\geq k_2+1.
	\end{array}
	\right.
\end{align*}
In all three cases, we reach a contradiction. Hence, $(i,j)=(1,2)$ and so $$(\mathcal{F}_1,\mathcal{F}_2)\simeq(\mathcal{C}_1(k_1,t), ~\mathcal{C}_2(k_2,t;k_1)).$$

\textbf{Case~3.} For $(i,j)\in\{(1,2),(2,1)\}$, $\mathcal{T}_i$ is not $t$-intersecting, and $\mathcal{T}_j$ is $t$-intersecting with $|\mathcal{T}_j|\geq 2$.

By Proposition \ref{l+1_t+1_t}, we have $k_i\geq t+2$. It follows from Lemma \ref{T=2-C} that $|\mathcal{T}_j|=2$ and $|\mathcal{T}_i|\leq k_i+1$.
If $|\mathcal{T}_i|\leq k_i$, by Lemmas~\ref{not-contain}, \ref{<hh}~(iii), \ref{<aa}~(iii) and \ref{<cc}~(ii), then
\begin{align*}
	|\mathcal{F}_1||\mathcal{F}_2|
<\left\{
	\begin{array}{ll}
        h(k_1,k_2,t)h(k_2,k_1,t),&\text{if}~t=1, \vspace{0.1cm}\\
        a(k_1,t)a(k_2,t),&\text{if}~t\geq 2~\text{and}~t+2\leq k_i\leq 2t+2, \vspace{0.1cm}\\
		c_1(k_i,t)c_2(k_j,k_i,t),&\text{if}~t\geq 2~\text{and}~k_i\geq 2t+3.
	\end{array}
	\right.
\end{align*}
If $|\mathcal{T}_i|=k_i+1$ and $k_j\geq t+2$, by Lemmas~\ref{T=2-C}~(ii) and \ref{lh<hh}, then
\begin{align*}
	|\mathcal{F}_1||\mathcal{F}_2|
<\left\{
	\begin{array}{ll}
		h(k_1,k_2,t)h(k_2,k_1,t),&\text{if}~t=1,~\text{or}~t\geq 2~\text{and}~k_i\geq 4t-1,\vspace{0.1cm}\\
        a(k_1,t)a(k_2,t),&\text{if}~t\geq 2~\text{and}~t+2\leq k_i\leq 4t-2.
	\end{array}
	\right.
\end{align*}
Each of the cases leads to a contradiction. Hence, $|\mathcal{T}_i|=k_i+1$ and $k_j=t+1$. Then $(i,j)=(1,2)$. It follows from Lemma~\ref{T=2-C}~(ii) that
$$(\mathcal{F}_1,\mathcal{F}_2)\simeq (\mathcal{H}(k_1,t;X,Y),~\mathcal{H}(t+1,t;Y ,X)),$$
where $X\in {[t+1,n]\choose k_1}$ and $Y\in {[t+1, n]\choose 2}$ satisfy $|X\cap Y|\geq 2-\delta_{1,t}$.

\textbf{Case~4.} Neither $\mathcal{T}_1$ nor $\mathcal{T}_2$ is $t$-intersecting.

By Lemma \ref{both-non-inter} (i), we have $(|\mathcal{T}_1|, |\mathcal{T}_2|)\in \big\{(2,2),~(2,3),~(3,2),~(3,3),~(2,4),~(4,2)\big\}$.
If $(|\mathcal{T}_1|, |\mathcal{T}_2|)\neq (3,3)$, by Lemmas \ref{not-contain} and \ref{<aa}~(ii),  then
\begin{align*}
|\mathcal{F}_1||\mathcal{F}_2|
&\leq g(|\mathcal{T}_2|,k_1,k_2,t) g(|\mathcal{T}_1|,k_2,k_1,t)\\ &\leq \max\big\{g(2,k_1,k_2,t)g(4,k_2,k_1,t),~ g(4,k_1,k_2,t)g(2,k_2,k_1,t)\big\}<a(k_1,t)a(k_2,t),
\end{align*}
a contradiction. Now suppose $(|\mathcal{T}_1|, |\mathcal{T}_2|)=(3,3)$. Since $|\mathcal{F}_1||\mathcal{F}_2|$ is maximized, it follows from Lemma~\ref{both-non-inter}~(ii) that $t=1$ and
$$(\mathcal{F}_1,\mathcal{F}_2) \simeq (\mathcal{B}(k_1;(1,3,2,4)),~\mathcal{B}(k_2;(1,2,3,4))).$$

This completes the proof of Theorem \ref{main-1}.$\qed$

\textbf{\emph{The proof of Theorem \ref{main-2}.}}
Assume that $\mathcal{F}\subseteq{[n]\choose k}$ is a maximal $t$-intersecting family such that $\tau_t(\mathcal{F})=t+1$. Let $\mathcal{T}$ be the collection of all $t$-covers of $\mathcal{F}$ with size $t+1$, and $M=\cup_{A\in \mathcal{T}}A$.
It follows from {\rm \cite[Lemma~2.1]{Cao-set}} that $\mathcal{T}$ is a $t$-intersecting family, and so $\tau_t(\mathcal{T})\in\{t,t+1\}$.

\textbf{Case~1.} $\tau_t(\mathcal{T})=t+1$.

From Fact~\ref{DRG}, there exists $M_0\in {[n]\choose t+2}$ such that $\mathcal{T}\subseteq {M_0\choose t+1}$, and there exist $A_1,A_2,A_3\in \mathcal{T}$ with $|A_1\cap A_2\cap A_3|\leq t-1$. Let $F\in \mathcal{F}$. Clearly, we know $|F\cap M_0|\geq t$. If $|F\cap M_0|=t$, then $F\cap M_0=F\cap A_1=F\cap A_2=F\cap A_3$, and so $|A_1\cap A_2\cap A_3|\geq t$, a contradiction. Hence, $|F\cap M_0|\geq t+1$, and thus ${M_0\choose t+1}\subseteq \mathcal{T}$, implying that $|\mathcal{T}|=t+2$. It follows from Lemma~\ref{equivalent} that $\mathcal{F}\simeq \mathcal{A}(k,t)=\mathcal{F}'$ in (i).

\textbf{Case~2.}  $\tau_t(\mathcal{T})=t$.

By Lemma \ref{m}~(i) with $\mathcal{F}=\mathcal{G}$, we have $|\mathcal{T}|\leq k-t+1$.

\textbf{Case~2.1.}  $|\mathcal{T}|=k-t+1$. Then $\mathcal{F}\simeq \mathcal{F}'$ in (ii) by Lemma \ref{m}~(iib).

\textbf{Case~2.2.} $|\mathcal{T}|=1$.
Then $\mathcal{T}=\{M\}$ and $[M\rangle_k\subseteq \mathcal{F}$ from \eqref{abs}. Since $\tau_t(\mathcal{F})=t+1$, we have $\mathcal{F}\setminus [M\rangle_k\neq \emptyset$ and $|F\cap M|=t$ for each $F\in \mathcal{F}\setminus [M\rangle_k$. Denote ${M\choose t}=\{T_i\}_{i\in [t+1]}$. Hence, $\mathcal{F}\setminus [M\rangle_k=\cup_{i\in [t+1]}\big(\mathcal{F}_{T_i}\setminus [M\rangle_k\big)$.

For all $i\in [t+1]$ and each $\mathcal{C}'_i\subseteq {[n]\setminus M\choose k-t}$, set $$U(\mathcal{C}'_i)=\big\{T_i~\cup~C : C\in \mathcal{C}'_i\big\}.$$
It is easy to verify that
$\cup_{i\in[t+1]}U(\mathcal{C}'_i)$ is $t$-intersecting if $\mathcal{C}'_1,\mathcal{C}'_2, \dots, \mathcal{C}'_{t+1}$ are pairwise cross intersecting.

Let
$$\mathcal{C}_i=\big\{F\setminus T_i : F\in \mathcal{F}_{T_i}\setminus [M\rangle_k\big\}.$$
Note that $\mathcal{C}_i\subseteq {[n]\setminus M\choose k-t}$ and $\mathcal{F}\setminus [M\rangle_k\subseteq \cup_{i\in[t+1]}U(\mathcal{C}_i)$.
It follows from $\tau_t(\mathcal{F})=t+1$ that at least two families of $\mathcal{C}_1,\mathcal{C}_2, \dots, \mathcal{C}_{t+1}$ are non-empty.
Let $C_i\in \mathcal{C}_i$ and $C_j\in \mathcal{C}_j$ with distinct $i,j\in[t+1]$. Then $T_i \cup C_i$ and $T_j \cup C_j$ are two elements of $\mathcal{F}$. Since
$$t\leq |(T_i \cup C_i)\cap (T_j \cup C_j)|=|T_i\cap T_j|+|C_i\cap C_j|=t-1+|C_i\cap C_j|,$$
we have $|C_i\cap C_j|\geq 1$, implying that $\mathcal{C}_i$ and $\mathcal{C}_j$ are cross intersecting. Hence, $\mathcal{C}_1,\mathcal{C}_2, \dots, \mathcal{C}_{t+1}$ are pairwise cross intersecting. Then $\cup_{i\in[t+1]}U(\mathcal{C}_i)$ is a $t$-intersecting family. Since $\mathcal{F}\setminus [M\rangle_k\subseteq \cup_{i\in[t+1]}U(\mathcal{C}_i)$, by the maximality of $\mathcal{F}$, we have $\mathcal{C}_1,\mathcal{C}_2, \dots, \mathcal{C}_{t+1}$ are maximal pairwise cross intersecting, and $\mathcal{F}=\cup_{i\in[t+1]}U(\mathcal{C}_i)\cup [M\rangle_k\simeq \mathcal{F}^{\prime}$ in (iii).

\textbf{Case~2.3.} $2\leq |\mathcal{T}|\leq k-t$.
Then there exists a $t$-subset $T$ such that $\mathcal{T}\subseteq [T\rangle_{t+1}$ by Fact \ref{DRG},  and $t+2\leq |M|\leq k$.
From \eqref{abs},  we obtain that
$$\cup_{A\in \mathcal{T}}[A\rangle_k=\left\{F\in {[n]\choose k} : T\subseteq F,~|F\cap M|\geq t+1\right\}\subseteq \mathcal{F}.$$
Let $F\in \mathcal{F}\setminus \cup_{A\in \mathcal{T}}[A\rangle_k$ and $A\in \mathcal{T}$.
Then $$t=|F\cap A|=|F\cap(T\cup (A\setminus T))|=|F\cap T|+|F\cap (A\setminus T)|.$$
It follows from $|A\setminus T|=1$ that either $|F\cap T|=t$ and $|F\cap (A\setminus T)|=0$,  or $|F\cap T|=t-1$ and $|F\cap (A\setminus T)|=1$.
If $|F\cap T|=t$, then $F\cap M=F\cap (\cup_{A\in \mathcal{T}}A)=T$, and so
$$F=T\cup C~\text{for~some}~C\in {[n]\setminus M\choose k-t}.$$
If $|F\cap T|=t-1$, then $M\setminus T=\cup_{A\in\mathcal{T}}(A\setminus T)\subseteq F$, and so
$$F=(M\setminus \{x\})\cup D~\text{for~some}~x\in T~\text{and}~D\in{[n]\setminus M\choose k-|M|+1}.$$

For each $\mathcal{C}'\subseteq {[n]\setminus M\choose k-t}$ and $\mathcal{D}'\subseteq {[n]\setminus M\choose k-|M|+1}$, set
$$U_1(\mathcal{C}')=\big\{T\cup C : C\in\mathcal{C}'\big\}~\text{and}~
U_2(\mathcal{D}')=\big\{(M\setminus \{x\})\cup D : x\in T,~ D\in\mathcal{D}'\big\}.$$
It is routine to check that $U_1(\mathcal{C}')\cup U_2(\mathcal{D}')$ is $t$-intersecting if $\mathcal{C}'$ and $\mathcal{D}'$ are cross intersecting.

Let
$$\mathcal{C}=\big\{F\setminus T : F\in \mathcal{F}_T\setminus \big(\cup_{A\in \mathcal{T}}[A\rangle_k\big)\big\}~\text{and}~\mathcal{D} =\big\{F\setminus M : F\in \mathcal{F}\setminus \mathcal{F}_T\big\}.$$
Observe that $\mathcal{C}\subseteq {[n]\setminus M\choose k-t}$, $\mathcal{D}\subseteq {[n]\setminus M\choose k-|M|+1}$ and $\mathcal{F}\setminus \big(\cup_{A\in \mathcal{T}}[A\rangle_k\big)\subseteq U_1(\mathcal{C})\cup U_2(\mathcal{D})$.
It follows from $\tau_t(\mathcal{F})=t+1$ that  $ \mathcal{D}\neq \emptyset$.
If  $\mathcal{C}=\emptyset$, then $\mathcal{C}$ and $\mathcal{D}$ are trivially cross intersecting. Suppose $\mathcal{C}\neq \emptyset$. Let $C\in \mathcal{C}$ and $D\in \mathcal{D}$. Then $C\cup T\in \mathcal{F}$, and $D\cup (M\setminus \{x\})\in \mathcal{F}$ for some $x\in T$. Since
$$t\leq |(C\cup T)\cap (D\cup (M\setminus \{x\}))|=|T\cap (M\setminus \{x\}) |+|C\cap D|=t-1+|C\cap D|,$$
we have $|C\cap D|\geq 1$, implying that $\mathcal{C}$ and $\mathcal{D}$ are cross intersecting.
Therefore, $U_1(\mathcal{C})\cup U_2(\mathcal{D})$ is $t$-intersecting. Since $\mathcal{F} \subseteq U_1(\mathcal{C})\cup U_2(\mathcal{D})\cup \big(\cup_{A\in \mathcal{T}}[A\rangle_k\big)$, by the maximality of $\mathcal{F}$, we have $\mathcal{F} = U_1(\mathcal{C})\cup U_2(\mathcal{D})\cup \big(\cup_{A\in \mathcal{T}}[A\rangle_k\big)$, and $\mathcal{C}$ and $\mathcal{D}$ are maximal cross intersecting.
Hence, $\mathcal{F}\simeq \mathcal{F}'$ in (iv).

This completes the proof of Theorem \ref{main-2}.$\qed$

\section{Some inequalities}\label{3}
In this section, we prove some inequalities used in the proof of Theorem \ref{main-1}. We always assume that $n$, $k$, $\ell$ and $t$ are positive integers with $k\geq t+1$, $\ell\geq t+1$ and $n\geq (t+1)^2(k+\ell)^2(k-t+1)(\ell-t+1)+k+\ell-t$.

\begin{lem}{\rm(\cite[Lemma~2.8]{Cao-Lu-Lv-Wang-2024})} \label{notknow}
Let $w$, $m$ and $s$ be non-negative integers.
\begin{enumerate}[{\rm(i)}]
  \item If $n\geq 2(k-t+1)(\ell-t+1)+t+1$ and $t\leq s<k$, then the function ${\ell-w\choose t-w}{n-s-t+w\choose k-s-t+w}$ is increasing in $w$ whenever $\max\{0,s+t-k\}\leq w\leq t-1$.
  \item If $n\geq (t+1)^2(k-t+1)(\ell-t+1)+t+1$, then $f(m,k,\ell,n,t)=k^{m-t-2}(k-t+1)^2{m\choose t}{n-m\choose \ell-m}$ is decreasing in $m$ whenever $t\leq m\leq \ell$, and $f(m,k,\ell,n,t)<(t+1)(k-t+1){n-t-1\choose \ell-t-1}$ if $t+2\leq m\leq \ell$.
\end{enumerate}
\end{lem}

Let $g(m,x,y,t)$, $a(x,t)$, $c_1(y,t)$, $c_2(x,y,t)$ and $h(x,y,t)$ be as in \eqref{equ-3}--\eqref{equ-2}. Set
\begin{align}
\widetilde{a}(x,t)&=a(x,t){n-t-1\choose x-t-1}^{-1},\label{equ-4}\\
\widetilde{h}(x,y,t)&=h(x,y,t){n-t-1\choose x-t-1}^{-1},\label{equ-5}\\
\widetilde{c}_1(y,t)\widetilde{c}_2(x,y,t)&= c_1(y,t)c_2(x,y,t){n-t-1\choose y-t-1}^{-1}{n-t-1\choose x-t-1}^{-1},\label{equ-7}\\
\widetilde{g}(m,x,y,t)&=g(m,x,y,t) {n-t-1\choose x-t-1}^{-1}.\label{equ-6}
\end{align}

\begin{lem}\label{cor-3.2}
The following inequalities hold.
\begin{enumerate}[{\rm(i)}]
  \item $\widetilde{h}(k,\ell ,t)\widetilde{h}(\ell,k,t) >(k-t+1)(\ell-t+1)-\frac{1}{t+1}.$
  \item $\widetilde{a}(k,t)\widetilde{a}(\ell,t) >\left(t+\frac{19}{10}\right)^2$.
  \item If $\ell\geq t+2$, then $1+\frac{1}{(k+\ell)^2(k-t+1)}> \widetilde{c}_1(\ell,t)\widetilde{c}_2(k,\ell,t) -(t+1)(\ell-t)>1-\frac{\ell-t}{(t+1)(k+\ell)^2}$.
  \item $\widetilde{g}(m,k,\ell,t) \widetilde{g}(m',\ell,k,t) \leq \left(m+\frac{1} {16}\right) \left(m'+\frac{1} {16}\right)$, where $m$ and $m'$ are positive integers.
\end{enumerate}
\end{lem}
\proof Let $x$ and $y$ be integers with $\min\{x,y\}\geq t+1$. By {\rm \cite[Proposition~1.6]{Frankl-Wang-2022}},
we obtain that
 \begin{align}\label{y-t}
 {n-y-1\choose x-t-1}\geq \left(1- \frac{(x-t-1)(y-t)}{n-t-1}\right){n-t-1\choose x-t-1}.
 \end{align}

(i) It follows from Pascal's formula and \eqref{y-t} that
\begin{align*}
\widetilde{h}(x,y,t)&=\left(\sum_{i=t+1}^{y+1}{n-i\choose x-t-1}+t\right){n-t-1\choose x-t-1}^{-1}\notag \\
&> (y-t+1){n-y-1\choose x-t-1}{n-t-1\choose x-t-1}^{-1}\notag\\
&\geq (y-t+1)\left(1-\frac{(x-t)(y-t)}{n-t-1}\right).
\end{align*}
Then
\begin{align*}
\widetilde{h}(k,\ell ,t)\widetilde{h}(\ell,k,t) &>(k-t+1)(\ell-t+1) \left(1-\frac{(k-t)(\ell-t)}{n-t-1}\right)^2\\
&>(k-t+1)(\ell-t+1) \left(1-\frac{2(k-t)(\ell-t)}{n-t-1}\right)\\
&\geq(k-t+1)(\ell-t+1) -\frac{1}{t+1},
\end{align*}
and thus (i) holds.

(ii) It follows from Pascal's formula and \eqref{equ-4} that
$$\widetilde{a}(x,t)=a(x,t){n-t-1\choose x-t-1}^{-1} =t+2-\frac{(t+1)(x-t-1)}{n-t-1}>t+\frac{19}{10},$$
whenever $n\geq 10(t+1)(x-t-1)+t+2$.
Then (ii) holds.

(iii) From Pascal's formula, we have $c_2(k,\ell,t)=(t+1)\sum_{i=t+2}^{\ell+1}{n-i\choose k-t-1}+{n-t-1\choose k-t-1}$, and so
\begin{align*}
\big((t+1)(\ell-t)+1\big){n-t-1\choose k-t-1}> c_2(k,\ell,t)>(t+1)(\ell-t){n-\ell-1\choose k-t-1}+{n-t-1\choose k-t-1}.
\end{align*}
As $\ell\geq t+2$ and $c_1(\ell,t){n-t-1\choose \ell-t-1}^{-1}=1+(t+1){n-t-1\choose \ell-t-1}^{-1}$, we have
\begin{align*}
1+\frac{t+1}{n-t-1}>c_1(\ell,t){n-t-1\choose \ell-t-1}^{-1}>1.
\end{align*}
Since $n\geq (t+1)^2(k+\ell)^2(k-t+1)(\ell-t+1)+k+\ell-t$, we can obtain
\begin{align*}
\widetilde{c}_1(\ell,t)\widetilde{c}_2(k,\ell,t) &<\left((t+1)(\ell-t)+1\right) \left(1+\frac{t+1}{n-t-1}\right)\\ &<(t+1)(\ell-t)+1+\frac{1}{(k+\ell)^2(k-t+1)},
\end{align*}
and the former part of (iii) holds.
From \eqref{y-t}, it follows that
\begin{align*}
\widetilde{c}_1(\ell,t)\widetilde{c}_2(k,\ell,t) &> (t+1)(\ell-t)\left(1-\frac{(k-t-1)(\ell-t)}{n-t-1}\right)+1\\
&>(t+1)(\ell-t)+1-\frac{\ell-t} {(t+1)(k+\ell)^2},
\end{align*}
and the latter part of (iii) holds.

(iv) By direct computation, if $\min\{x,y\}\geq t+1$ and $n\geq (t+1)^2(x+y)^2(x-t+1)(y-t+1)+x+y-t$, then
\begin{align}\label{gg}
\widetilde{g}(m,x,y,t) &=m+\frac{y(t+1)(y-t+1)(x-t-1)}{n-t-1}\leq m+\frac{y}{(t+1)(x+y)^2}.
\end{align}
Since $x\geq t+1\geq 2$, we have $\frac{y}{(t+1)(x+y)^2}\leq \frac{y}{2(y+2)^2}$. The function $\frac{y}{2(y+2)^2}$ is monotonically decreasing for $y\geq 2$. Hence, $\widetilde{g}(m,x,y,t)\leq m+\frac{1}{16}$, and (iv) holds.
$\qed$

\begin{lem}\label{<hh}
Let $\ell\geq t+2$. The following inequalities hold.
\begin{enumerate}[{\rm(i)}]
  \item $g(2(\ell-t+1),k,\ell,t)g(1,\ell,k,t) < h(k,\ell,t)h(\ell,k,t)$.
  \item $g(\ell-t,k,\ell,t)g(k-t+1,\ell,k,t) < h(k,\ell,t)h(\ell,k,t)$.
  \item $g(\ell,k,\ell,1)g(2,\ell,k,1) < h(k,\ell,1)h(\ell,k,1)$.
\end{enumerate}
\end{lem}
\proof
By \eqref{equ-5} and \eqref{equ-6}, it suffices to show $\widetilde{h}(k,\ell,t) \widetilde{h}(\ell,k,t)> \widetilde{g}\left(m,k,\ell,t\right) \widetilde{g}(m',\ell,k,t),$
where $m$ and $m'$ are related integers in (i), (ii) and (iii).

(i)
Suppose $k\geq t+2$.
It follows from Lemma \ref{cor-3.2} (i) and (iv) that
$\widetilde{h}(k,\ell,t)\widetilde{h}(\ell,k,t) -\widetilde{g}\left(2(\ell-t+1) ,k,\ell,t\right)\widetilde{g}(1,\ell,k,t)>(k-t-\frac{9}{8})(\ell-t+1)-
\frac{1}{t+1}- \frac{17}{256}>0.$

Suppose $k=t+1$.  By \eqref{gg}, we have $$\widetilde{g}\left(2(\ell-t+1),t+1,\ell,t\right)=2(\ell-t+1) ~\text{and}~\widetilde{g}(1,\ell,t+1,t)\leq 1+\frac{1}{(\ell+t+1)^2}.$$
By \eqref{equ-5}, we have $\widetilde{h}(t+1,\ell,t)=\ell+1$,
and
$$\widetilde{h}(\ell,t+1,t)
>2{n-t-2\choose \ell-t-1}{n-t-1\choose \ell-t-1}^{-1}
= 2\left(1-\frac{\ell-t-1}{n-t-1}\right) \geq 2-\frac{1}{4(\ell+t+1)^2},$$
provided that $n\geq 2(t+1)^2(\ell+t+1)^2(\ell-t+1)+\ell+1$.
Hence
\begin{align*}
(\ell+1)\left(2-\frac{1}{4(\ell+t+1)^2}\right) -2(\ell-t+1)\left(1+\frac{1}{(\ell+t+1)^2}\right)= 2t-\frac{9(\ell+1)-8t}{4(\ell+t+1)^2}
>0,
\end{align*}
implying (i) holds.

(ii)
By \eqref{gg}, we have
\begin{align*}
&~~~~\widetilde{g}(\ell-t,k,\ell,t) \widetilde{g}(k-t+1,\ell,k,t)\\
&\leq(k-t+1)(\ell-t)+\frac{\ell(k-t+1)}{(t+1)(k+\ell)^2} +\frac{k(\ell-t)}{(t+1)(k+\ell)^2} +\frac{k\ell}{(t+1)^2(k+\ell)^4}\\
&< (k-t+1)(\ell-t)+\frac{2}{t+1}+\frac{1}{(t+1)^2(k+\ell)^2}.
\end{align*}

Hence, from Lemma \ref{cor-3.2} (i) and $k\geq t+1$, it follows that
\begin{align*}
&~~~~\widetilde{h}(k,\ell,t) \widetilde{h}(\ell,k,t)- \widetilde{g}(\ell-t,k,\ell,t) \widetilde{g}(k-t+1,\ell,k,t)\\
&> k-t+1-\frac{3}{t+1}-\frac{1}{(t+1)^2(k+\ell)^2}>0,
\end{align*}
and (ii) holds.

(iii) Note that $t=1$, and so $k\geq 2$ and $\ell\geq 3$.

Suppose $k\geq 3$. Then from \eqref{equ-3} and (ii),  it follows that $$g(\ell,k,\ell,1)g(2,\ell,k,1)\leq g(\ell,k,\ell,1)g(k-1,\ell,k,1)<h(k,\ell,1)h(\ell,k,1).$$

Suppose $k=2$. Then $n\geq 8\ell(\ell+2)^2+\ell+1$. By \eqref{gg}, we have $\widetilde{g}(\ell,2,\ell,1)=\ell$ and $\widetilde{g}(2,\ell,2,1)\leq 2+\frac{1}{(\ell+2)^2}.$
By the proof of (i), we obtain that
$\widetilde{h}(2,\ell,1)\widetilde{h}(\ell,2,1) >(\ell+1)\left(2-\frac{1}{4(\ell+2)^2}\right)$.
Since
\begin{align*}
(\ell+1)\left(2-\frac{1}{4(\ell+2)^2}\right) -\ell\left(2+\frac{1}{(\ell+2)^2}\right)= 2-\frac{5\ell+1}{4(\ell+2)^2}>0,
\end{align*}
we have $\widetilde{h}(2,\ell,1)\widetilde{h}(\ell,2,1) >\widetilde{g}(\ell,2,\ell,1)\widetilde{g}(2,\ell,2,1)$. Then (iii) follows.$\qed$

\begin{lem}\label{<aa}
The following inequalities hold.
\begin{enumerate}[{\rm(i)}]
  \item $g(t+2,k,\ell,t)g(t+1,\ell,k,t) < a(k,t)a(\ell,t)$.
  \item $g(4,k,\ell,t)g(2,\ell,k,t) < a(k,t)a(\ell,t)$.
  \item $g(\ell,k,\ell,t)g(2,\ell,k,t) < a(k,t)a(\ell,t)$ for $\ell\leq 2t+2$.
\end{enumerate}
\end{lem}
\proof If $(m,m')=(t+2,t+1)$, $(m,m')=(4,2)$, or $(m,m')=(\ell,2)$ with $\ell\leq 2t+2$, then by Lemma \ref{cor-3.2} (ii) and (iv), we obtain
\begin{align*}
\widetilde{a}(k,t)\widetilde{a}(\ell,t) -\widetilde{g}(m,k,\ell,t)\widetilde{g}(m',\ell,k,t) >\left(t+\frac{19}{10}\right)^2 -\left(m+\frac{1}{16}\right)\left(m'+\frac{1}{16}\right)>0,
\end{align*}
and so the lemma follows from \eqref{equ-4} and \eqref{equ-6}.$\qed$

\begin{lem}\label{<cc}
Let $\ell\geq t+2$. The following inequalities hold.
\begin{enumerate}[{\rm(i)}]
  \item $g((t+1)(\ell-t),k,\ell,t)g(1,\ell,k,t) < c_1(\ell,t)c_2(k,\ell,t)$.
  \item $g(\ell,k,\ell,t)g(2,\ell,k,t) < c_1(\ell,t)c_2(k,\ell,t)$ for $t\geq 2$ and $\ell\geq 2t+3$.
\end{enumerate}
\end{lem}
\proof
(i)
By \eqref{gg}, we have
\begin{align*}
\widetilde{g}((t+1)(\ell-t),k,\ell,t) \widetilde{g}(1,\ell,k,t)
\leq(t+1)(\ell-t)+\frac{\ell +(t+1)k(\ell-t)} {(t+1)(k+\ell)^2}+\frac{k\ell} {(t+1)^2(k+\ell)^4}.
\end{align*}
Then by Lemma \ref{cor-3.2} (iii), we obtain that
\begin{align*}
&~~~~(t+1)(k+\ell)^2\cdot\big(\widetilde{c}_1(\ell,t)
\widetilde{c}_2(k,\ell,t)- \widetilde{g}((t+1)(\ell-t),k,\ell,t) \widetilde{g}(1,\ell,k,t)\big)\\
&\geq (t+1)(k+\ell)^2-(t+1)k(\ell-t)-2\ell+t -\frac{k\ell}{(t+1)(k+\ell)^2}>0,
\end{align*}
and so (i) holds from \eqref{equ-7} and \eqref{equ-6}.

(ii)
From Lemma~\ref{cor-3.2}~(iii) and (iv), it follows that
\begin{align*}
\widetilde{c}_1(\ell,t)\widetilde{c}_2(k,\ell,t)- \widetilde{g}(\ell,k,\ell,t) \widetilde{g}(2,\ell,k,t)
>\left(t-\frac{17} {16}\right)\ell-t(t+1)
>0,
\end{align*}
due to $t\geq 2$ and $\ell\geq 2t+3$. Then (ii) holds from \eqref{equ-7} and \eqref{equ-6}.$\qed$

\begin{lem}\label{lh<hh}
Let $\ell\geq t+2$. The following hold.
\begin{enumerate}[{\rm(i)}]
  \item If $k\geq t+2$ and $\ell\geq 4t-1$, then $(\ell+1){n-t-1\choose k-t-1}h(\ell,t+1,t) < h(k,\ell,t)h(\ell,k,t)$.
  \item If $t\geq 2$ and $\ell\leq 4t-2$, then $(\ell+1){n-t-1\choose k-t-1}h(\ell,t+1,t) < a(k,t)a(\ell,t)$.
\end{enumerate}
\end{lem}
\proof
Since $\ell\geq t+2$ and $n\geq (t+1)^2(k+\ell)^2(k-t+1)(\ell-t+1)+k+\ell-t$, we obtain from \eqref{equ-5} that
$$\widetilde{h}(\ell,t+1,t)<\left(2{n-t-1\choose \ell-t-1}+t\right){n-t-1\choose \ell-t-1}^{-1}\leq 2+\frac{t}{n-t-1}<\frac{21}{10}
.$$

(i) By Lemma \ref{cor-3.2} (i), we have
\begin{align*}
\widetilde{h}(k,\ell ,t)\widetilde{h}(\ell,k,t)-(\ell+1)\widetilde{h}(\ell,t+1,t)
&> (k-t+1)(\ell-t+1)-\frac{1}{t+1}-\frac{21}{10}(\ell+1)
>0,
\end{align*}
due to $k\geq t+2$ and $\ell\geq 4t-1$. Then (i) holds from \eqref{equ-5}.

(ii) By Lemma \ref{cor-3.2} (ii), we have
\begin{align*}
\widetilde{a}(k,t)\widetilde{a}(\ell,t) -(\ell+1)\widetilde{h}(\ell,t+1,t)
>\left(t+\frac{19}{10}\right)^2 -\frac{21}{10}(\ell+1)>0
\end{align*}
due to $\ell\leq 4t-2$. Then (ii) holds from \eqref{equ-4} and \eqref{equ-5}. $\qed$

\begin{lem}\label{cc<hac}
The following hold.
\begin{enumerate}[{\rm(i)}]
  \item If $k\geq 2t$, then $c_1(k,t)c_2(k,k,t)< \big(h(k,k,t)\big)^2$.
  \item If $t+2\leq k\leq 2t-1$, then $c_1(k,t)c_2(k,k,t)< \big(a(k,t)\big)^2$.
  \item If $\ell\geq k+1$, then $c_1(k,t)c_2(\ell,k,t)<c_1(\ell,t)c_2(k,\ell,t)$.
\end{enumerate}
\end{lem}
\proof
(i) Since $k\geq 2t$, it follows from Lemma \ref{cor-3.2} (i) and (iii) that
\begin{align*}
&~~~~(\widetilde{h}(k,k,t))^2- \widetilde{c}_1(k,t)\widetilde{c}_2(k,k,t)\\
&>(k-t+1)^2-\frac{1}{t+1}- (t+1)(k-t)-1-\frac{1}{4k^2(k-t+1)}\\
&=(k-2t+1)(k-t)-\frac{1}{t+1}-\frac{1}{4k^2(k-t+1)}>0,
\end{align*}
and (i) holds from \eqref{equ-5} and \eqref{equ-7}.

(ii) Since $t+2\leq k\leq 2t-1$, it follows from Lemma~\ref{cor-3.2} (ii) and (iii) that
\begin{align*}
(\widetilde{a}(k,t))^2- \widetilde{c}_1(k,t)\widetilde{c}_2(k,k,t) &>\left(t+\frac{19}{10}\right)^2 -(t+1)(k-t)-1-\frac{1}{4k^2(k-t+1)}
>0.
\end{align*}
Then (ii) holds from \eqref{equ-4} and \eqref{equ-7}.

(iii) It follows from Lemma \ref{cor-3.2} (iii) that
\begin{align*}
&~~~~\widetilde{c}_1(\ell,t)\widetilde{c}_2(k,\ell,t)- \widetilde{c}_1(k,t)\widetilde{c}_2(\ell,k,t)\\
&> (t+1)(\ell-t)+1-\frac{\ell-t}{(t+1)(k+\ell)^2} -(t+1)(k-t)-1-\frac{1}{(k+\ell)^2(\ell-t+1)}\\
&\geq t+1-\frac{1}{(t+1)(k+\ell)} -\frac{1}{(k+\ell)^2(\ell-t+1)}>0,
\end{align*}
and so (iii) follows from \eqref{equ-7}. $\qed$

\section*{Acknowledgments}
B. Lv is supported by the National Natural Science Foundation of China (12571347 \& 12131011), and Beijing Natural Science Foundation (1252010).  K. Wang is supported by the National Natural Science Foundation of China (12131011 \& 12571347) and Beijing Natural Science Foundation (1252010 \& 1262010).

\end{document}